\documentclass[11pt]{amsart}

\usepackage{subfiles}

\usepackage{tabu, fullpage,diagbox, booktabs}
\usepackage{placeins}
\usepackage[margin=1in]{geometry}
\usepackage[dvipsnames]{xcolor}   		             		
\usepackage{graphicx}			
\usepackage{amssymb}
\usepackage{mathrsfs}
\usepackage{amsthm}
\usepackage{amsmath}
\usepackage{stmaryrd}
\usepackage{tikz}
\usepackage{tikz-cd}
\usepackage{accents}
\usepackage{upgreek}
\usepackage{enumerate}
\usepackage{bm}
\usepackage{mathtools}
\usepackage[all]{xy}
\usepackage{caption}

\usepackage{url}
\usepackage{float}
\usepackage{bbm}

\usepackage[full]{textcomp}
\usepackage[cal=cm]{mathalfa}
\usepackage{xparse}
\usepackage{comment}
\usepackage{cite}

\tikzset{
	commutative diagrams/.cd, 
	arrow style=tikz, 
	diagrams={>=stealth}
}

\theoremstyle{definition}

\makeatletter
\def\@tocline#1#2#3#4#5#6#7{\relax
	\ifnum #1>\c@tocdepth 
	\else
	\par \addpenalty\@secpenalty\addvspace{#2}%
	\begingroup \hyphenpenalty\@M
	\@ifempty{#4}{%
		\@tempdima\csname r@tocindent\number#1\endcsname\relax
	}{%
		\@tempdima#4\relax
	}%
	\parindent\z@ \leftskip#3\relax \advance\leftskip\@tempdima\relax
	\rightskip\@pnumwidth plus4em \parfillskip-\@pnumwidth
	#5\leavevmode\hskip-\@tempdima
	\ifcase #1
	\or\or \hskip 1em \or \hskip 2em \else \hskip 3em \fi%
	#6\nobreak\relax
	\dotfill\hbox to\@pnumwidth{\@tocpagenum{#7}}\par
	\nobreak
	\endgroup
	\fi}
\makeatother

\usetikzlibrary{calc}
\usetikzlibrary{fadings}
\usetikzlibrary{decorations.pathmorphing}
\usetikzlibrary{decorations.pathreplacing}

\newcounter{marginnote}
\DeclareMathAlphabet{\mathpzc}{OT1}{pzc}{m}{it}

\usepackage[backref=page]{hyperref}
\hypersetup{
	colorlinks   = true,          
	urlcolor     = blue,          
	linkcolor    = purple,          
	citecolor   = blue             
}

\theoremstyle{definition}
\newtheorem{theorem}{Theorem}[section]
\newtheorem{computation}[theorem]{Computation}

\newtheorem{lemma}[theorem]{Lemma}
\newtheorem{proposition}[theorem]{Proposition}
\newtheorem{remark}[theorem]{Remark}

\newtheorem*{runningexample*}{Running example}

\newtheorem*{aside*}{Aside}

\newtheorem{constr}[theorem]{Construction}

\newtheorem{definition}[theorem]{Definition}
\newtheorem{example}[theorem]{Example}

\newtheorem{proposition-definition}[theorem]{Proposition-Definition}

\newenvironment{construction}    
{%
	\pushQED{\qed}\begin{constr}}
	{\popQED\end{constr}}

\newcommand{\PTnXbeta}{\mathrm{\mathcal{PT}}_n}
\newcommand{\UM}{\mathcal{X}}
\newcommand{\M}{\mathcal{PT}_0}
\newcommand{\VM}{{PT}_0(X,\beta)}
\newcommand{\VTM}{|\Delta|}
\newcommand{\VUM}{{X}_{\mathrm{uni}}}
\newcommand{\VTUM}{{X}^{\mathfrak{t}}_{\mathrm{uni}}}
\newcommand{\Sp}{\mathscr{P}_\Delta}
\newcommand{\VSp}{\mathscr{P}_\Delta}
\newcommand{\TSp}{\mathscr{P}_\Delta^\mathfrak{t}}
\newcommand{\R}{\mathscr{X}}
\newcommand{\TR}{\mathscr{X}^{\mathfrak{t}}}

\newcommand{\TPN}{(\mathbb{P}^m)^{\mathfrak{t}}}
\newcommand{\TX}{{X}^{\mathfrak{t}}}
\newcommand{\B}{{\mathcal{P}}}
\newcommand{\VB}{{{P}}}
\newcommand{\TB}{{\B}^{\mathfrak{t}}}
\newcommand{\VTB}{{\VB}^{\mathfrak{t}}}
\newcommand{\HH}{N}
\newcommand{\RH}{\mathscr{N}}

\DeclareRobustCommand{\stirling}{\genfrac\{\}{0pt}{}}

\newcommand{\bcd}{\begin{center}\begin{tikzcd}}
	\newcommand{\ecd}{\end{tikzcd}\end{center}}

\usepackage{accents}

\newcommand{\vardbtilde}[1]{\tilde{\raisebox{0pt}[0.85\height]{$\tilde{#1}$}}}

\NewDocumentCommand{\compatibilitydatum}{m m m m m m O{} O{} O{}}{
	\begin{equation\star} \begin{tikzcd}[ampersand replacement=\&]
		\: \arrow{r} \& {#1} \arrow{r} \arrow{d}{#7} \& {#2} \arrow{r} \arrow{d}{#8} \& {#3} \arrow{r}{[1]} \arrow{d}{#9} \& \: \\
		\: \arrow{r} \& {#4} \arrow{r} \& {#5} \arrow{r} \& {#6} \arrow{r} \& \:
		\end{tikzcd} \end{equation\star}}

\NewDocumentCommand{\commutingsquare}{m m m m o O{} O{} O{} O{}}{
	\begin{equation}\begin{tikzcd}[ampersand replacement=\&] \label{#5}
	#1 \arrow{r}{#6} \arrow{d}{#7} \& #2 \arrow{d}{#8} \\
	#3 \arrow{r}{#9} \& #4
	\end{tikzcd}\IfValueTF{#5}{\label{#5}}{} \end{equation}}

\NewDocumentCommand{\cartesiansquare}{m m m m O{} O{} O{} O{}}{
	\begin{equation\star}\begin{tikzcd}[ampersand replacement=\&]
		#1 \arrow{r}{#5} \arrow{d}{#6} \arrow[dr, phantom, "\square"] \& #2 \arrow{d}{#7} \\
		#3 \arrow{r}{#8} \& #4
		\end{tikzcd} \end{equation\star}}

\NewDocumentCommand{\cartesiansquarelabel}{m m m m m O{} O{} O{} O{}}{
	\begin{tikzcd}[ampersand replacement=\&]
	#1 \arrow{r}{#6} \arrow{d}{#7} \arrow[dr, phantom, "\square"] \& #2 \arrow{d}{#8} \\
	#3 \arrow{r}{#9} \& #4
	\end{tikzcd}\IfValueTF{#5}{\label{#5}}{}
}

\NewDocumentCommand{\triangleofspaces}{m m m O{} O{} O{}}{
	\begin{tikzcd} [ampersand replacement=\&]
	#1 \arrow{r}{#4} \arrow[bend right]{rr}{#5} \& #2 \arrow{r}{#6} \& #3
	\end{tikzcd}}

\begin{document}
	
	\title{Logarithmic Pandharipande--Thomas Spaces and the Secondary Polytope}
	\author{Patrick Kennedy--Hunt}
	
	\begin{abstract}
		Maulik and Ranganathan have recently introduced moduli spaces of logarithmic stable pairs. In the case of toric surfaces we recast this theory using three ingredients: Gelfand, Kapranov and Zelevinsky secondary polytopes, Hilbert schemes of points, and tautological vector bundles. In particular, logarithmic stable pairs spaces are expressed as the zero set of an explicit section of a vector bundle on a logarithmically smooth space, thus providing an explicit description of their virtual fundamental class. A key feature of our construction is that moduli spaces are completely canonical, unlike the existing construction, which is only well-defined up to logarithmic modifications. We calculate the Euler--Satake characteristics of our moduli spaces in a number of basic examples.  These computations indicate the complexity of the spaces we construct.
	\end{abstract}
	
	\maketitle
	\tableofcontents
	\pagebreak
	
	\pagebreak
	\section*{Introduction}
	
	Let $\Delta$ be a two dimensional lattice polytope in $M_X=\mathbb{Z}^2$. This datum determines a proper toric surface $X$ equipped with an ample curve class $\beta$. We consider $X$ as a logarithmic scheme equipped with divisorial logarithmic structure from its toric boundary. In Section \ref{sec:one} we construct the \textit{logarithmic linear system} which is a toric stack $\M(X,\beta)$.  There is a natural birational toric morphism to the linear system $|\beta|$, $$\M(X,\beta)\xrightarrow{\pi_{|\beta|}} |\beta|.$$ The logarithmic linear system is a moduli space of curves on \textit{expansions} of $X$ and is closely related to two constructions:
	\begin{enumerate}[(1)]
		\item The secondary fan associated to $\Delta$. This is the normal fan to the secondary polytope introduced by Gelfand, Kapranov and Zelevinsky \cite{GKZ}.
		\item Logarithmic Donaldson--Thomas spaces introduced by Maulik and Ranganathan in \cite{MR20}.
	\end{enumerate}
	\noindent The spaces in \cite{MR20} are only well--defined up to a class of (virtual) birational modifications. Our spaces are canonical and, as we explain below, are a terminal object in the system of birational models discussed in \cite[Section 3]{MR20}. 
	
	The Pandharipande--Thomas theory associated to a variety $X$, integer $n$ and curve class $\beta$ concerns the moduli space of \textit{stable pairs} on $X$ with discrete data $(n,\beta)$\cite{PT,PT2}. If $X$ is a surface, the moduli space of stable pairs coincides with the relative Hilbert scheme of points $\mathrm{Hilb}^n(\textbf{C}/|\beta|)$. Here $\textbf{C} \rightarrow |\beta|$ is the universal curve over the linear system $|\beta|$. Bootstrapping our construction of the logarithmic linear system, we obtain canonical logarithmic Pandharipande--Thomas moduli spaces for $X$ a toric surface. We denote these logarithmic analogues of $\mathrm{Hilb}^n(\textbf{C}/|\beta|)$ by $\PTnXbeta(X,\beta)$. In the sequel we typically suppress the data of $X$ and $\beta$, writing $\PTnXbeta(X,\beta)=\PTnXbeta$.
	
	\subsection{Main Results}
	There are three steps in our construction of logarithmic Pandharipande--Thomas spaces: \begin{enumerate}[(1)]
		\item Build a torus equivariant diagram of toric stacks from a diagram of fans.
		\item Pass to the relative Hilbert scheme of points and restrict attention to an open subset.
		\item Pass to a closed subscheme cut out by a section of a tautological vector bundle.
	\end{enumerate} 
	In related work, Maulik and Ranganathan define the logarithmic Pandharipande--Thomas moduli space for a curve class on a threefold \cite[Remark 4.5.2]{MR20}.
	\begin{theorem}\label{thm:main}
		The moduli space $\PTnXbeta$ is a proper, logarithmic Deligne--Mumford stack equipped with a universal diagram:
		$$
		\begin{tikzcd}
		& \UM_n \arrow[d, "\varpi_n"] \arrow[r, "\pi_X"] & X                        \\
		& \PTnXbeta \arrow[r, "\mathrm{ev}"']           & \mathrm{Hilb}^{\mathrm{log}}_\beta(\partial X)
		\end{tikzcd}.
		$$
		The fibers of the map $\varpi_n$ are logarithmically smooth surfaces equipped with a map to $X$. There is a universal stable pair on $\mathcal{X}_n$ denoted $$[\mathcal{O}_X \xrightarrow{s} F],$$ we identify $\mathrm{ker}(s)$ with the ideal sheaf of a subscheme $\mathcal{Z}_n$ of $\mathcal{X}_n$. The universal subscheme $\mathcal{Z}_n$ is a curve strongly transverse to the boundary of $\mathcal{X}_n$ in the sense of \cite[Section 0.2]{MR20}. Let $\Delta^o$ be the set of lattice points in the interior of $\Delta$ then $\varpi_n^{-1}(p)\cap \mathcal{Z}_n$ has holomorphic Euler characteristic $2-2|\Delta^o|+n$. The moduli space $\M$ is a toric stack.
	\end{theorem}

	We write $\mathcal{Z}_0 =\mathcal{Z}$ and $\varpi_0 = \varpi$. Fixing a point $x$ in $\M$ the fibre $X_{\Gamma_x}=\varpi^{-1}(x)$ is a logarithmically smooth surface and $$\mathcal{Z}_x~=~\varpi^{-1}~(x)~\cap~\mathcal{Z}$$ a curve on this surface. The intersection of $\mathcal{Z}_x$ with the logarithmic boundary of $X_{\Gamma_x}$ satisfies a transversality condition. The morphism $\mathrm{ev}$ records the intersection of $\mathcal{Z}_x$ with the boundary of $X_{\Gamma_x}$.
	\subsection{Topology of stable pairs spaces}
	Euler--Satake characteristics are a version of Euler characteristic for orbifolds defined independently in \cite{MR95520,MR1435975}. The Euler--Satake characteristic of a toric stack is a count with multiplicity of maximal cones in the associated fan. We understand the Euler--Satake characteristic of logarithmic stable pairs spaces as a weighted count of cones in their tropicalisation. The weight assigned to each cone is a linear combination of Euler--Satake characteristics for relative Hilbert schemes of points without boundary. 
	
	In the case $X = \mathbb{P}^1\times \mathbb{P}^1$ and for the $(1,d)$ curve class we compute the Euler--Satake characteristic. The results for small values of $n$ and $d$ are recorded in Table \ref{tab:P1P1EulerChar}. To provide an explicit formula requires a great deal of notation so a statement is deferred to Theorem \ref{thm:EulerChar}.
	
	\begin{table}[h]
		\begin{tabular}{@{}c|cccc@{}}
			\diagbox[width=2.5em]{n}{$\beta$}&  $(1,1)$  & $(1,2)$ & (1,3) & (1,4) \\ \midrule
			0	& 20 & 20 & 456 & 32 \\
			1	& 96  & 224  & 9524 & 56576 \\
			2	& 400  & 1155  & 52023 & -865699 \\
			3   & $1218\frac{2}{3}$  & $7982\frac{2}{3}$  & 313644$\frac{1}{3}$ &  -9379411$\frac{1}{3}$
		\end{tabular}
		\caption{}\label{tab:P1P1EulerChar}
	\end{table}
	Euler characteristics of Pandharipande--Thomas spaces are closely related to the Gromov--Witten and Pandharipande--Thomas invariants of surfaces \cite{Kool_2014,MR3238155}. In particular, Euler characteristics of Pandharipande--Thomas spaces capture information about the locus of points in the linear system corresponding to curves with a fixed number of nodes \cite{KoolShendeThomas}. The closure of the locus of curves with $d$ nodes is denoted $S_d$ and called a \textit{Severi variety}. Fixing the degree and number of nodes in a curve fixes the geometric genus. Understanding Euler--Satake characteristics of logarithmic Pandharipande--Thomas spaces thus provides enumerative information. 
	
	\subsection{Intersection theory}
	The simplest logarithmic stable pairs space associated to $X$ is the logarithmic linear system $\M$. Recall $\M$ is a toric stack. Intersection theory of toric varieties has a clean formalism in the language of Minkowski weights due to work of Fulton and Sturmfels \cite{IntToricVar}. We exploit this in Section \ref{sec:computations} to compute enumerative invariants. These enumerative invariants are integrals of insertions over the fundamental class of the logarithmic linear system. Insertions in logarithmic Pandharipande--Thomas theory are recalled in Section \ref{sec:two:ptinsert}. Denote the class in the Chow cohomology ring $A^\star(|\beta|)$ dual to a hyperplane by $[H]$.  
	\begin{proposition}\label{thm:identifyminkwt}
		The insertion $\tau_0([pt])$ coincides with $\pi_{|\beta|}^\star([H])$.
	\end{proposition}
	This intuitive result allows us to compute first examples. There is a fan associated to $\M$, each cone $\sigma$ in the fan defines a cycle in the Chow group $A_\star(\M)$. For $\sigma$ a cone of codimension $k$ we give a combinatorial algorithm to compute $\int \tau_{0}([\mathrm{pt}])^k[V(\sigma)]$. 
	We explain how to apply logarithmic insertions and with this calculus compute the following maximal contacts situation in Section \ref{sec:firstcomp}.
	
	\begin{computation}\label{comp:curvesinP1}
		Take $X = \mathbb{P}^2$ and $\beta$ the class of curves degree $d$. Let $\mu$ be the class in $\mathrm{Hilb}^{\mathrm{log}}_d(\partial \mathbb{P}^2)$ of ideal sheaves supported on a single point of (an expansion of) each $\mathbb{P}^1$ in the boundary of $X$. In this situation,
		$$\int (\tau_0([\mathrm{pt}]))^{\frac{(d-4)(d-1)}{2}}\mathrm{ev}^\star(\mu) = 1.$$
	\end{computation} 
	
	The exponent $(d-4)(d-1)/2$ is chosen such that the expected dimension of the integrand the dimension of the linear system. Assuming the logarithmic Gromov--Witten/ Donaldson--Thomas correspondence \cite{MR23}, there is a link between this type of intersection theory on secondary polytopes and Mikhalkin's theorem \cite{Milk}. We intend to explore this connection, and how it can be used to understand the tropicalisation of the Severi variety, in future work.
	
	\subsection{Conventions} We adopt the convention that all toric varieties are normal. We often have cause to work with the fan associated to a toric variety. The fan of a toric variety $V$ will be written $V^\mathfrak{t}$; where $V$ is not a toric variety $V^\mathfrak{t}$ denotes the tropicalisation. What we mean by tropicalisation will be clear from context. All schemes are of finite type over $\mathbb{C}$. All logarithmic structures are fine and saturated.
	
	\subsection{Relation to other work}
	
	The moduli spaces constructed in \cite{MR20} differ from their analogues in Pandharipande--Thomas theory and logarithmic Gromov--Witten theory \cite{abramovich2011stable,chen2011stable,gross2012logarithmic} in that the spaces are only well--defined up to birational modification. The situation is parallel to the spaces of expanded stable maps constructed in \cite{Ran19}, where again, there is no minimal model in general. Our results here show the existence of a minimal model in the special case of a toric surface. 
	
	A link between tropical and algebraic curve counts was pioneered by Mikhalkin \cite{Milk} for surfaces, generalised to higher dimension targets by Nishinou--Siebert\cite{NishSieb06} and linked to logarithmic Gromov--Witten theory by Mandel and Ruddat \cite{MR4068259}. Bousseau proved that certain tropical curve counts could be expressed in terms of logarithmic Gromov Witten invariants for toric surfaces with $\lambda_g$ insertions \cite{BOSS}. Assuming a logarithmic version of the Gromov--Witten/Donaldson--Thomas conjecture \cite{MNOP1,MNOP2}, Bousseau noted \cite[Section~1.1.4]{BOSS} his computation represents a prediction about intersection theory on logarithmic Pandharipande--Thomas spaces. 
	
	The technique of studying curves in toric degenerations has a rich precident in enumerative geometry \cite{BOSS, MR4068259,MR20,NishSieb06, Ran19}. The original construction is due to Mumford\cite{CM_1972__24_3_239_0}. A version of the idea behind the logarithmic linear system in the context of Gromov--Witten theory was suggested by Katz \cite{Katz}. More precisely, Katz was interested in using the geometry of the secondary polytope to impose conditions on enumerative geometry problems.
	
	\subsection{Future Directions and generalisations}
	Theorem~\ref{thm:main}, specifically the existence of a terminal model, holds for any surface $X$ equipped with a simple normal crossing divisor. Thus a canonical logarithmic Pandharipande--Thomas moduli space exists in this general case, opening the way to a number of computations and basic results. The case of logarithmic Calabi--Yau pairs is of particular interest.
	
	The {logarithmic linear system} $\mathcal{PT}_0(X,\beta)$ of divisors on a toric surface generalises to define the logarithmic linear system of divisors of fixed class on a proper toric variety $X$ of any dimension. The constructions of this paper carry over vis a vis to define a similarly well behaved space whose combinatorics are related to secondary polytopes of higher dimension. These spaces are the first examples of components of the logarithmic Hilbert schemes defined in \cite{KH23a} which are not cases of the \cite{MR20} construction.
	
	Studying moduli spaces of subschemes of toric varieties of higher codimension is also possible. The combinatorics in this situation are related to Chow polytopes defined by Kapranov, Sturmfels and Zelevinsky \cite{MR1174606}.
	
	\subsection{Acknowledgements}
	The author wishes to thank Dhruv Ranganathan for suggesting the topic, many helpful conversations and feedback on drafts of this paper. The author is also greatful to the anonymous referee for feedback which improved the presentation of ideas. Thanks are due to Navid Nabijou for several helpful conversations, especially for suggesting our approach in Section~\ref{sec:algorithm}. Dmitri Whitmore is owed thanks for useful discussions on how to implement the algorithm from Section~\ref{sec:algorithm} in Python. Finally thanks are due to the members of the Cambridge algebraic geometry group for the excellent atmosphere.
	
	\section{Fans, toric varieties and the tropical linear system}\label{sec:TropLinSys}
	Let $\Delta$ be a two dimensional lattice polytope in a two dimensional lattice $M_X = \mathbb{Z}^2$. Associated to any polytope there is a fan structure on the dual lattice called the \textit{normal fan}, see \cite[Chapter~5, Section~4B]{GKZ}. Dimension $k$ cones of the normal fan to polytope $P$ biject with codimension $k$ faces of $P$. The normal fan of $\Delta$ is a fan structure $X^\mathfrak{t}$ on the lattice $N_X = \mathrm{Hom}(M_X,\mathbb{Z})$. The identification of $M_X$ with $\mathbb{Z}^2$ gives an identification of $N_X$ with $\mathbb{Z}^2$ by taking the dual basis to the standard basis of $M_X = \mathbb{Z}^2$. 
	
	In this section we work exclusively with fans. The central theme is understanding the connection between maps of fans constructed from secondary fans and tropical curves. We orient the reader by noting the output of this section will be two morphisms of fans $$ \VTUM \xrightarrow{\varpi^\mathfrak{t}} |\Delta| \textrm{ and }\TR\xrightarrow{\kappa^\mathfrak{t}} \TSp. $$  We reserve script letters for objects pertaining to $\kappa$. We construct $|\Delta|$, which is the fan of the coarse moduli space of $\M$, as a subdivision of $\TR$ and therefore begin by understanding the second morphism.
	
	\subsection{The secondary fan}
	We now recall the construction of the \textit{secondary fan of} $\Delta$. This is the normal fan to the secondary polytope of $\Delta$ \cite[Chapter 7 Section 1.C]{GKZ}.
	
	\subsubsection{Ambient lattices}  Define $m= \mathrm{card}(\Delta)-1$ and identify functions from (the integral points in) $\Delta$ to $\mathbb{Z}$ with $\mathbb{Z}^{m+1}$. Elements of the monoid $\HH \cong \mathbb{Z}^m$ thus biject with equivalence classes of functions $f:\Delta\rightarrow \mathbb{Z}$ where $f,f'$ are identified if $f-f'$ is constant. 
	Define functions from $\Delta$ to $\mathbb{Z}$ $$e_x:(i,j) \mapsto i \textrm{ and } e_y:(i,j) \mapsto j,$$ since both functions are linear, there is a natural identification of monoids $N_X = \langle e_x,e_y\rangle$ and we define $\mathscr{N} = N/N_X.$
	
	\subsubsection{Induced polytope subdivisions} The convex hull of the points of $\Delta$ is written $\mathrm{supp}(\Delta)$. A function $\varphi:\Delta \rightarrow \mathbb{R}$ is the data of an element of $N_\mathbb{R} =N \otimes \mathbb{R}$. There is a unique piecewise linear function $\hat{\varphi}: \mathrm{supp}(\Delta) \rightarrow \mathbb{R}$ with graph the lower convex hull of $$\{((i,j),\varphi(i,j))|(i,j)\in \Delta\}\textrm{ in } \mathrm{supp}(\Delta) \times \mathbb{R}.$$ The bend locus of $\hat{\varphi}$ defines a polyhedral subdivision $\Delta_\varphi$ of $\mathrm{supp}(\Delta)$. Note $\Delta_\varphi$ depends only on the image of $\varphi$ in $\mathscr{N}_\mathbb{R} = \mathscr{N}\otimes \mathbb{R}$.
	
	\subsubsection{The secondary fan}\label{sec:SecondaryFan} The \textit{secondary fan} $\TSp$ is the coarsest fan structure on $\mathscr{N}_\mathbb{R}$ such that whenever $\varphi_1, \varphi_2$ lie in the interior of the same cone, there is an equality $\Delta_{\varphi_1} = \Delta_{\varphi_2}.$ 
	The linear system $|\beta| \cong \mathbb{P}^m$ induces an unrelated fan structure on $N_\mathbb{R}$. The $m+1$ maximal cones of $|\beta|^\mathfrak{t}$ are indexed by $\Delta$; the maximal cone labelled $(i,j)$ is the collection of functions $\varphi$ in $N$ such that $\varphi$ adopts its minimal value on $(i,j)$ in $\Delta$.
	We define a fan $\TR$ as the coarsest subdivision of $|\beta|^\mathfrak{t}$ such that every function in a fixed cone induces the same polyhedral decomposition $\Delta_\varphi$. Write $\kappa^\mathfrak{t}$ for the natural map $\TR \rightarrow \TSp.$
	
	\subsection{The tropical linear system} In this section we understand points of $\TSp$ as specifying tropical curves. We upgrade this observation to construct a fan $|\Delta|$ which is a tropical version of the linear system $|\beta|$ mentioned in the introduction.
	\subsubsection{Tropical curves}
	A \textit{tropical curve} $\Gamma_\varphi$ for $\varphi$ in $N_\mathbb{R}$, is the collection of points $(x,y)$ in $N_X\otimes \mathbb{R}$ on which the minimum $$\mathrm{min}_{(i,j)\in \Delta}\{ix+jy + \varphi(i,j)\}$$ is achieved at least twice. An abstract tropical curve $[\Gamma_\varphi]$ is an equivalence class of tropical curves under the smallest equivalence relation identifying $\Gamma_\varphi$ with $\Gamma_{\varphi'}$ whenever $\varphi-\varphi'$ lies in $N_X$. Restricting $\varphi_\mathbb{R}$ to $\Delta$ defines a function $\varphi$ and thus a subdivision $\Delta_\varphi$. The \textit{combinatorial type} $\mathcal{C}(\Gamma_\varphi)$ of $\Gamma_\varphi$ is the subdivision $\Delta_\varphi$ of $\Delta$. Intuitively, two tropical curves have the same combinatorial type if they are qualitatively the same, differing only in the lengths and proportions of their edges.
	\begin{remark}\label{rem:AbsTropCurveTranslates}
		Note $N_X \otimes \mathbb{R} = \mathbb{R}^2$ acts on tropical curves. An element $(c_1,c_2)\in \mathbb{R}^2$ sends $$\Gamma_\varphi \mapsto \Gamma_{\varphi+c_1e_x+c_2e_y} = \Gamma_\varphi + (c_1,c_2).$$ Thus the equivalence class in the definition abstract tropical curve corresponds to the translation action of $N_X \otimes \mathbb{R}$.
	\end{remark}
	\subsubsection{Pre--expansion tropical curves} 
	Note a tropical curve $\Gamma_\varphi$ is the same data as a polyhedral subdivision $\mathfrak{P}_{\Gamma_\varphi}$ of $N_X\otimes \mathbb{R}$ - the tropical curve $\Gamma_\varphi$ is the one skeleton of $\mathfrak{P}_{\Gamma_\varphi}$. Given two tropical curves $\Gamma_{\varphi_1},\Gamma_{\varphi_2}$ we obtain a new polyhedral subdivision $\mathfrak{P}_{\Gamma}$ of $N_X\otimes \mathbb{R}$ by taking the common refinement of $\mathfrak{P}_{\Gamma_{\varphi_1}}$ and $\mathfrak{P}_{\Gamma_{\varphi_2}}$. We denote the resulting one skeleton $\Gamma$. A \textit{pre--expansion tropical curve} is any set $\Gamma$ obtained by setting $\Gamma_{\varphi_2}$ to be the fan $\TX = \Gamma_{0}$ of $X$. 
	
	The underlying topological space of $\Gamma_\varphi$ is denoted $\underline{\Gamma}_\varphi$. The \textit{combinatorial type $\mathcal{C}(\Gamma)$ of a pre--expansion tropical curve} $\Gamma$ obtained by superimposing $\Gamma_\varphi$ onto the fan of $X$ is the pair $$(\mathcal{C}(\Gamma_\varphi), \iota:\underline{\Gamma}_\varphi \hookrightarrow  \underline{\Gamma})$$ where $\iota$ is the inclusion of topological realisations. The definition of combinatorial type captures when two tropical curves are qualitatively the same. In the sequel we reserve $\Gamma_\varphi$ for tropical curves and use $\Gamma$ to denote a pre--expansion tropical curve, unless stated otherwise.
	
	\subsubsection{Flatness for morphisms of fans} A morphism of fans is said to be \textit{combinatorially flat} if the image of every cone is an entire cone. Note combinatorial flatness does not imply the corresponding map of toric varieties is flat.
	
	For us a \textit{family of tropical curves} is a combinatorially flat morphism of complete fans of relative dimension two. This notion of a family of tropical curves admits a  generalisation \cite[Section~2]{KH23a} to the setting of cone complexes and more general tropical objects.
	
	\subsubsection{Family of abstract tropical curves}  
	Fix a function $\varphi^\dagger$ in $\mathscr{N}_\mathbb{R}$ and choose a lift $\varphi_0$ in $N_\mathbb{R}$. For any $\varphi$ in $(\kappa^\mathfrak{t})^{-1}(\varphi^{\dagger})$ we may write $\varphi = \varphi_0 + ae_x + be_y$. Identify $(\kappa^\mathfrak{t})^{-1}(\varphi^\dagger)$ with $N_X\otimes \mathbb{R}$ via the map $\varphi\mapsto (a,b)$ in $N_X\otimes \mathbb{R}$. Thinking of $N_X \otimes \mathbb{R}$ as a subspace of $N_\mathbb{R}$ this map sends a point $\varphi$ of $(\kappa^\mathfrak{t})^{-1}(\varphi^{\dagger})$ to $\varphi - \varphi_0$. Lemma \ref{lem:auxcomb} justifies the terminology family of abstract tropical curves.
	
	\begin{lemma}\label{lem:auxcomb}
		Restricting the fan $\TR$ to $\varphi_0 + N_X\otimes \mathbb{R}\cong N_X\otimes \mathbb{R}$ defines the polyhedral subdivision $\mathfrak{P}_{\Gamma_{\varphi_0}}$ of $N_X\otimes \mathbb{R}$.
	\end{lemma}
	The choice of $\varphi_0$ in the preimage of $\varphi^\dagger$ does not effect the abstract tropical curve $[\Gamma_{\varphi_0}]$. Changing $\varphi_0$ corresponds to taking a different representative of the equivalence class.
	\begin{proof}
		Consider the polyhedral structure $\mathfrak{P}$ on $\varphi_0 + N_X\otimes \mathbb{R}\cong N_X\otimes \mathbb{R}$ induced by $\TR$. The fan $\TR$ is a subdivision of $\TPN$. All points of $(\kappa^\mathfrak{t})^{-1}(\varphi^\dagger) \cong N_X\otimes \mathbb{R}$ induce the same subdivision $\Delta_{\varphi^\dagger}$ of $\Delta$ so the polyhedral structure $\mathfrak{P}$ is induced from the fan structure of $|\beta|^\mathfrak{t}$. Moreover, the preimage of $\varphi$ is not contained in the $m-1$ skeleton of $|\beta|^\mathfrak{t}$. Thus a point of $N_X\otimes \mathbb{R}$ lies in the one skeleton of $\mathfrak{P}$ if and only if it lies in the $m-1$ skeleton of $|\beta|^\mathfrak{t}$. In Section \ref{sec:SecondaryFan} we characterised the $m-1$ skeleton of $|\beta|^\mathfrak{t}$ as the collection of $\varphi$ achieving its minimum on (at least) two points of $\Delta$. Equivalently, a point $\varphi_0 + (a,b)$ of $\varphi_0 + N_X\otimes \mathbb{R}$ lies in the one skeleton if and only if $\varphi_0(i,j) +  ai  + b j$ achieves its minimum for two $(i,j)$ in $\Delta$. Thus the one skeleton is precisely $\Gamma_{\varphi_0}$.
	\end{proof}	
	
	\begin{example}\label{ex:secondpoly}
		Set $X= \mathbb{P}^1 \times \mathbb{P}^1$ and $\beta = (1,1)$ corresponding to a polytope $$\Delta = \{(0,0), (1,0), (0,1), (1,1)\}.$$ Here $m=3$ and the secondary fan $\TSp$ is necessarily the unique complete fan of dimension one so coincides with $(\mathbb{P}^1)^\mathfrak{t}$. To form $\TR$ we take the minimal subdivision of $(\mathbb{P}^3)^\mathfrak{t}$ such that the induced subdivision $\Delta_\varphi$ is constant as $\varphi$ varies within a cone. The map to $(\mathbb{P}^1)^\mathfrak{t}$ is taking dot product with the vector $(1,1,-1)$ so we subdivide along the plane $x+y-z=0$. Note the polyhedral subdivision on $N_X\otimes \mathbb{R}^2$ arising from the fibres over points of $(\mathbb{P}^1)^\mathfrak{t}$ are the tropical curves shown in Figure \ref{fig:Ts}.
	\end{example}
	
	\begin{figure}
		\includegraphics[width=\linewidth]{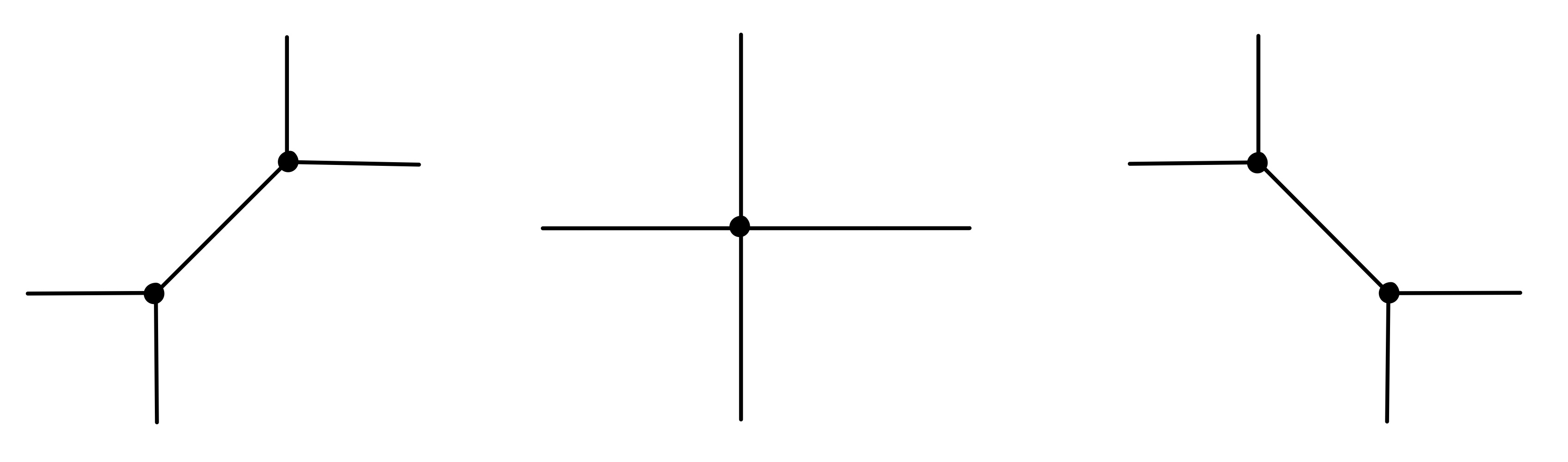}
		\caption{The three combinatorial types of tropical curve dual to the polytope corresponding to the $(1,1)$ class on $\mathbb{P}^1\times \mathbb{P}^1$. These correspond to the three cones of $(\mathbb{P}^1)^\mathfrak{t}$ as in Example \ref{ex:secondpoly}. The central diagram is the tropical curve over the zero cone.}\label{fig:Ts}
	\end{figure} 
	
	\subsubsection{Pre--expansion tropical curves from points of $\TR$}
	Let $W$ be the set of equivalence classes of pairs $(p,\Gamma_\varphi)$ where $p$ is a point in $N_X\otimes \mathbb{R}=\mathbb{R}^2$, $\Gamma_\varphi$ is a tropical curve, and the equivalence relation relates pairs $(p,\Gamma_\varphi)\sim (p',\Gamma_{\varphi'})$ whenever $(p,\Gamma_\varphi)$ is sent to $(p',\Gamma_{\varphi'})$ by the action of $\mathbb{R}^2$ on itself by addition. Elements of $W$ biject with pre--expansion tropical curves: we think of $p$ as specifying the location of the origin. 
	
	Points in the fan $\TR$ specify an element of $W$ by Lemma \ref{lem:auxcomb}. We have thus defined a map of sets $$f:N_\mathbb{R}\rightarrow \{\textrm{pre--expansion tropical curves}\}.$$
	\subsubsection{Augmented tropical curves}\label{sec:AugTropCurve} Fix a tropical curve $\Gamma_\varphi$. Take the fan $X^\mathfrak{t}$ and invert in the origin $$x \mapsto -x$$ to define a new fan $X_\mathfrak{t}$. Superimpose the fan $X_\mathfrak{t}$ at each vertex of $\Gamma_\varphi$. The result is the one skeleton of a polyhedral subdivision $\mathfrak{P}_{\Gamma_\varphi^\mathrm{aug}}$ of $N_X\otimes \mathbb{R}$ which we call the \textit{augmented subdivision}. The augmented subdivision of an abstract tropical curve $[\Gamma_\varphi]$ is the equivalence class of subdivisions $\mathfrak{P}_{\Gamma_\varphi^\mathrm{aug}}$ where the equivalence relation is translation by $N_X\otimes \mathbb{R}$.
	
	The combinatorial type of the pre--expansion tropical curve arising from an element $(p,\Gamma_\varphi)$ of $W$ depends on more data than the combinatorial type of $\Gamma_\varphi$ and the stratum of $\mathfrak{P}_{\Gamma_\varphi}$ in which $p$ lies. This combinatorial type is however specified by the stratum of $\mathfrak{P}_{\Gamma_\varphi^{\mathrm{aug}}}$ in which $p$ lies.
	
	\subsubsection{A family of pre-expansion tropical curves}\label{sec:FanLLS}
	
	We have established an assignment of points in the fan $\TR$ to pre--expansion tropical curves. Our next goal is to build a combinatorially flat morphism of fans to $\TR$ whose fibre over a point $\varphi$ of $N_\mathbb{R}$ carries polyhedral subdivision with one skeleton $f(\varphi)$. There can be no such combinatorially flat morphism to $\TR$ if $f$ assigns different combinatorial types to points in the interior of a single cone of $\TR$. 
	
	We are thus forced to subdivide $\TR$ as follows. To a face $e$ of $\Delta$ with outward facing normal vector $(a,b)$ associate the element $\varphi_e=ae_x+be_y$ of $N_\mathbb{R}$. We define a fan $|\Delta|$ as a subdivision of $\TR$ by subdividing along the codimension one locus of functions $\varphi$ which may be written as $g + \lambda \varphi_e$ for some edge $e$, positive real $\lambda$ and function $g$ adopting its minimal value on at least three points of $\Delta$ which are not colinear. To understand why this is the right subdivision to make, note that $g$ is a vertex of a tropical curve under the correspondence of Lemma \ref{lem:auxcomb}. Thus thinking of $N_X$ as a sublattice of $N$, functions of the form $g + \lambda \varphi_e$ correspond to functions $h$ such that $h-g$ lies on a ray of the fan of $X$.
	\begin{remark}
		One could instead construct the correct universal family over $\TR$ and observe it was not combinatorially flat. To resolve this one would push forward the fan structure on the universal family as in \cite[Proof of Lemma 2.1.1]{molcho2019universal}. The resulting subdivision is precisely $|\Delta| \rightarrow \TR$.
	\end{remark}
	
	Consider the fibre product in the category of fans $\VTM\times_{\TSp} \TR$. See \cite[Section 2.2]{molcho2019universal} for fibre products of fans. Points in this fan are pairs of functions $(\varphi,\phi)$ such that $\varphi-\phi$ lies in $N_X$. Pull back the fan structure of $X$ along the map 
	\begin{align*}
	\VTM\times_{\TSp} \TR &\rightarrow  N_X\\
	(\varphi,\phi) &\mapsto \varphi-\phi
	\end{align*} 
	to obtain a fan $\VTUM$ subdividing $\VTM~\times_{\TSp}~\TR.$ The pull back fan structure is defined to be the coarsest subdivision such that the image of each cone lies within a single cone of $X^\mathfrak{t}$. 
	
	Projection to the first factor induces a morphism ${\varpi}^\mathfrak{t}:\VTUM \rightarrow \VTM$; by construction we have a morphism $\pi_X^\mathfrak{t}: \VTUM \rightarrow \TX$.
	The next proposition makes a link between $\VTUM \rightarrow |\Delta|$ and a family of pre--expansion tropical curves.
	\begin{proposition}\label{prop:tropmod} With notation as above we prove the following.
		\hfill \begin{enumerate}[(1)]
			\item The combinatorial type of $f(\varphi)$ is constant for $\varphi$ in the interior of cones of $\VTM$. Moreover, fixing a combinatorial type $[\Gamma]$, there is a unique cone $\sigma$ such that whenever $f(\varphi)$ has combinatorial type $[\Gamma]$, in fact $\varphi$ lies in $\sigma$.
			\item Fix $\varphi$ in $N_\mathbb{R}$. Identifying $N_X\otimes \mathbb{R}$ with $(\varpi^\mathfrak{t})^{-1}(\varphi)$ via the map $\pi_X^\mathfrak{t}$, the polyhedral structure on $(\varpi^\mathfrak{t})^{-1}(\varphi)$ induced by the fan $\VTUM$ is $\mathfrak{P}_{f(\varphi)}$.
		\end{enumerate} 
	\end{proposition}
	\begin{proof}\hfill
		\begin{enumerate}[(1)]
			\item The map $\TR\rightarrow |\Delta|$ is combinatorially flat by construction. Consequently the combinatorial type of $f(\varphi)$ is constant as $\varphi$ moves within the interior of a cone. We now prove the cone $\sigma$ exists. Fixing $\mathcal{C}(\Gamma_\varphi)$, the combinatorial type of an augmented subdivision is a map of topological spaces $$\iota:\underline{\Gamma}_\varphi\rightarrow \underline{\Gamma}.$$ There is a cone $\sigma^\dagger$ of $\TSp$ such that $\varphi^\dagger$ lies in $\sigma$ whenever $\varphi^\dagger$ is a point of $\mathscr{N}_\mathbb{R}$ and $[\Gamma_{\varphi^\dagger}]$ has the same combinatorial type as $[\Gamma_{\varphi}]$. The preimage of $\sigma$ in $N_\mathbb{R}$ is a union of cones indexed by faces of $\mathfrak{P}_{\Gamma_{\varphi}}^\mathrm{aug}$. There is a unique face of this polyhedral subdivision corresponding to the combinatorial type $(\Gamma_\varphi,\iota)$. The cone indexed by this face is the cone $\sigma$ in our statement.
			\item We have identified the fibre over a point $\varphi$ in ${N}_\mathbb{R}$ of the projection map $\TR \times \VTM \rightarrow \VTM$ with $N_X \otimes \mathbb{R}$ by sending $(\varphi+ae_x + b e_y,\varphi)$ to $(a,b)$ in $N_X \otimes \mathbb{R}$. According to Lemma \ref{lem:auxcomb}, the polyhedral structure on this fibre arising from the fan $\TR \times \VTM$ has one skeleton given by those functions $(\varphi+ae_x + b e_y,\varphi)$ such that $(a,b)$ lies in $\Gamma_\varphi$. 
			
			The fan $\VTUM$ is obtained by subdividing the product $\TR \times \VTM \rightarrow \VTM$ by pulling back the fan structure of $\TX$ along $\pi^\mathfrak{t}_X$. The effect on the polyhedral structure of the fibre over $\varphi$ is to impose the fan of $X$ centred at the point $(\varphi,\varphi)$. 
		\end{enumerate}
	\end{proof}
	\begin{example}\label{ex:TropNotBiject}
		The morphism $f$ is not injective and so the cone complexes underlying the fans that we construct are not precisely the moduli spaces studied in \cite[Section 2]{KH23a}. The consequences are discussed in Remark \ref{rem:toricStructures}.
		To see $f$ is not injective consider functions $$\Delta = \mathrm{conv}\{(0,0), (2,0), (0,2)\}\xrightarrow{\varphi} \mathbb{R} $$
		which are nowhere negative and $\varphi(0,0) = \varphi(2,0)= \varphi(0,2) = 0$. All such $\varphi$ give rise to the same tropical curve $\Gamma_\varphi$.
	\end{example}
	\begin{figure}
		\includegraphics[width=\linewidth]{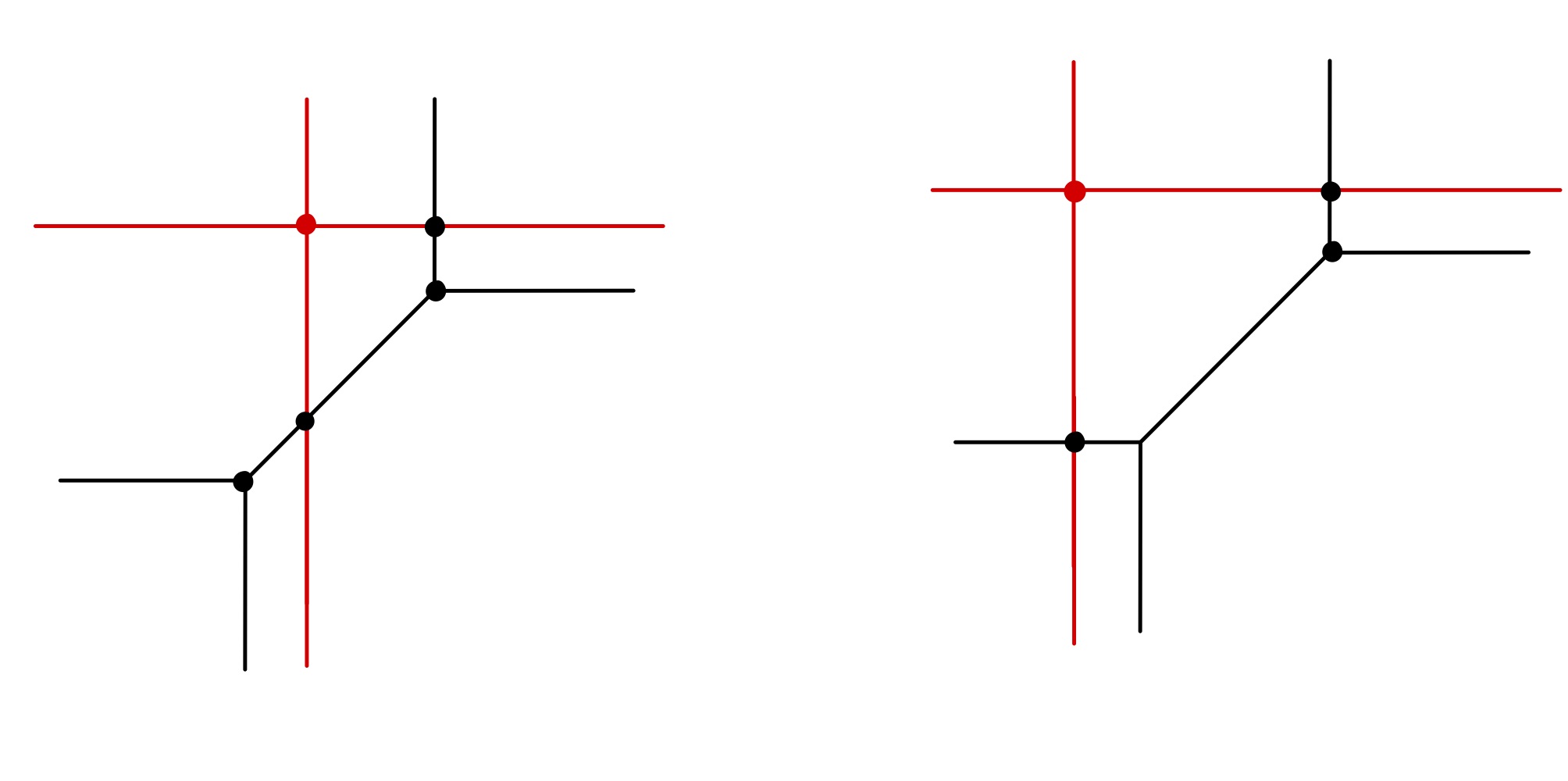}
		\caption{Two possible locations for the origin on the polyhedral subdivision arising from the first curve in Figure \ref{fig:Ts}. The two curves shown have different combinatorial type explaining why the subdivision is needed at the start of Section \ref{sec:FanLLS}.}\label{fig:twocone}
	\end{figure}
	
	\section{The logarithmic linear system}\label{sec:one}
	There is a functor from the category of fans to the category of (normal) toric varieties. We refer to this functor as the \textit{toric dictionary}. In this section we use the toric dictionary to build logarithmic moduli spaces and associated universal data from the morphisms of fans studied in Section~\ref{sec:TropLinSys}.
	\subsection{Construction of the logarithmic linear system}\label{sec:one:construction}
	Let $\Delta$ be a two dimensional lattice polytope in a two dimensional lattice $M_X \cong \mathbb{Z}^2$. These data determine a proper toric surface $X$ equipped with an ample curve class $\beta$. In this section we construct the following diagram discussed in the introduction.
	
	\begin{equation*}
	\begin{tikzcd}
	{\mathcal{Z}} \arrow[r, hook] & \UM_0 \arrow[d, "\varpi"] \arrow[r, "\pi_X"] & X                        \\
	& \M \arrow[r, "\mathrm{ev}"']           & \mathrm{Hilb}^{\mathrm{log}}_\beta(\partial X)
	\end{tikzcd}
	\end{equation*}
	In this diagram $\mathrm{Hilb}^{\mathrm{log}}_\beta(\partial X)$ is the logarithmic Hilbert scheme of points on the boundary of $X$. A point of $\mathrm{Hilb}^{\mathrm{log}}_\beta(\partial X)$ specifies a collection of points on an expansion of the logarithmic boundary of $X$.
	
	The diagram of fans on the left may be passed through the toric dictionary to obtain the diagram of toric varieties on the right. \begin{equation*}
	\begin{tikzcd}
	\VTUM \arrow[r] \arrow[d, "\varpi^\mathfrak{t}"] & \TX & \VUM \arrow[d, "\tilde{\varpi}"] \arrow[r] & X \\
	{|\Delta|}                                         &     & \VM                    &  
	\end{tikzcd}
	\end{equation*} 
	The universal expansion $\VUM$ contains a universal curve $Z$ which we now construct. 
	
	\begin{construction}\label{cons:univsubsch}
		(Universal subscheme.) We have constructed morphisms of fans $$\VTUM\rightarrow \VTM\times_{\TSp}\TR \rightarrow\TR \xrightarrow{\pi_{\mathbb{P}^m}^\mathfrak{t}} \TPN.$$ The toric dictionary gives corresponding morphisms of toric varieties $$\VUM\rightarrow Y \rightarrow\mathscr{X} \xrightarrow{\pi_{\mathbb{P}^m}} \mathbb{P}^m$$ We think of $\mathbb{P}^m$ as the linear system $|\beta|$ and thus identify coordinates $X_{i,j}$ with elements $(i,j)$ of $\Delta$. Pull back the line bundle and section $(\mathcal{O}(1),\sum_{(i,j)\in \Delta} X_{i,j})$ along $\pi_{\mathbb{P}^m}$ to give $(L', s')$. Define $Z^\dagger$ to be the zero set of $s'$. Define $Z$ to be the strict transform of $Z^\dagger$ in $X_{\mathrm{uni}}$.
	\end{construction}
	
	\subsubsection{Flattening the universal expansion}\label{sec:flattening} We perform universal semistable reduction to the morphism $\tilde{\varpi}$ to obtain a semistable (and thus flat) morphism of toric stacks. The tools were introduced in \cite{molcho2019universal}. The result is that $\varpi$ is a flat map in the following diagram.
	\begin{equation}\label{square}
	\begin{tikzcd}
	\mathcal{Z} \arrow[r] & \UM_0 \arrow[r, "F_{\UM}"] \arrow[d, "\varpi"] & \VUM \arrow[d, "\tilde{\varpi}"] \arrow[r] & X \\
	& \M \arrow[r, "F_{\M}"]             & \VM                      &  
	\end{tikzcd}
	\end{equation}
	
	\noindent The two parallel horizontal morphisms are maps from stacks to their coarse moduli spaces. Here $\mathcal{Z}$ is the strict transform of $Z$ under $F_{\UM}$.
	
	\subsubsection{Logarithmic Hilbert schemes of points} It remains to construct a morphism to the logarithmic Hilbert scheme of points on the boundary. To do so we build $\B_k$, the logarithmic linear system of $k$ points on $\mathbb{P}^1$. 
	
	Setting $\Delta=[0,k]$ a one dimensional polytope in a one dimensional lattice $M_{\mathbb{P}^1}$, the construction $\VUM \rightarrow \VM$ goes through vis a vis except now $\HH$ has dimension $k$, and $\RH$ is the quotient of $N$ by the rank $1$ subgroup spanned by $e_x$ where $$e_x: i \mapsto i.$$ The map $N \rightarrow N_X$ is replaced by a map $N \rightarrow \mathbb{Z}= N_{\mathbb{P}^1}$ along which we pull back the fan structure $(\mathbb{P}^1)^\mathfrak{t}$. The resulting diagram appears on the left below. Applying universal weak semistable reduction we obtain the diagram on the right
	$$
	\begin{tikzcd}
	X_{\mathrm{uni},\mathbb{P}^1} \arrow[r, "\pi_X'"] \arrow[d, "\tilde{\varpi}"] & \mathbb{P}^1 & \UM_{\mathbb{P}^1} \arrow[d, "\varpi"] \arrow[r, "\pi_X"] & \mathbb{P}^1 \\
	\VB_k                                         &   & \B_k                                        &  
	\end{tikzcd}.
	$$
	
	Label the dimension one faces of $\Delta$ as $\{1,...,p\}$ and suppose face $i$ has length $k_i$. Define $\mathrm{Hilb}^{\mathrm{log}}_\beta(\partial X) = \prod_{i=1}^p \B_{k_i}$. On the level of fans there are $p$ morphisms $\mathrm{ev}_i:\VTM\rightarrow \TB_{k_i}$ defined by restricting a function $\varphi$ in $N$ to the points of $\Delta$ on face $i$. These define morphisms in the algebraic category via the toric dictionary. Take the product of these maps to define the morphism $\mathrm{ev}$.
	
	\subsection{Expansions from tropical curves}\label{sec:one:modexp} 
	
	In this subsection we take the first steps to explaining the modular interpretation of $\M$. We first explain how to obtain an expansion of $X$ from a tropical curve. See \cite{InvitToricDegen} for background on toric degenerations. In the next construction we use $\Gamma$ to denote the one skeleton of any polyhedral subdivision of $\mathbb{R}^2$. In particular $\Gamma$ need not be a pre--expansion tropical curve.
	\begin{construction}\label{cons:expansions}
		The cone $C_\Gamma^\mathfrak{t}$ over $\mathfrak{P}_\Gamma$ gives rise via the toric dictionary to a toric variety equipped with a map to $\mathbb{A}^1$. One obtains a surface $X_\Gamma$ as the special fibre of this degeneration.
		
		Assume now $\mathfrak{P}_\Gamma$ refines the fan of $X$. For example, this occurs when $\Gamma$ is a pre--expansion tropical curve. There is a morphism of fans $C_\Gamma^\mathfrak{t}\rightarrow C_{X^\mathfrak{t}}^\mathfrak{t}$ and thus a morphism $X_\Gamma \rightarrow X.$
	\end{construction}
	
	\begin{definition}
		An \textit{expansion} of $X$ is the output of Construction~\ref{cons:expansions} when the input is a pre--expansion tropical curve. A \textit{pre--expansion} is the domain of an expansion.
	\end{definition}
	
	\begin{remark}\label{rem:SurfacesAndTropicalCurves}We explain the link to Section \ref{sec:one:construction}. Suppose $\varpi: Y\rightarrow X$ is a surjective morphism of proper toric varieties corresponding to combinatorially flat morphism of fans. Suppose further $\mathrm{dim}(Y) = \mathrm{dim}(X)+2$. 
		
		Fix $\sigma$ a cone in the fan of $X$. The preimage under $\varpi^\mathfrak{t}$ of a point $p$ of $\sigma$ is a copy of $N_X \otimes \mathbb{R}$ equipped with a polyhedral subdivision $\mathfrak{P}_{\Gamma_\sigma}$ induced by the fan structure of $Y$. Combinatorial flatness ensures the polyhedral subdivision is independent of the choice of $p$ in the interior of $\sigma$. 
		For $x$ any point in the torus orbit associated to $\sigma$, $\varpi^{-1}(x)\rightarrow X$ is the expansion arising from $\Gamma_\sigma$ through Construction \ref{cons:expansions}.
	\end{remark}
	In Section \ref{sec:FanLLS} we constructed families of pre--expansion tropical curves. We now understand that after applying the toric dictionary and universal weak semistable reduction, these maps yield morphisms we can think of as \textit{families of expansions}.
	\FloatBarrier
	\subsection{Strongly transverse and stable curves}\label{sec:unividealsheaf}
	
	In this subsection we explain which curves $\M$ parameterises. The key input is a theorem of Ascher and Molcho \cite[Theorem 1]{aschermolcho} which built on results of Kapranov, Sturmfels and Zelevinksy \cite{QuotientsToricVar}.
	
	\subsubsection{Strongly transverse and stable curves} Assume a pre--expansion tropical curve $\Gamma$ is obtained by superimposing tropical curve $\Gamma_\varphi$ on to the fan of $X$. There is a canonical morphism $q_\Gamma:X_{\Gamma} \rightarrow X_{\Gamma_{\varphi}}$ induced by the subdivision $$C_\Gamma^\mathfrak{t}\rightarrow C_{\Gamma_{\varphi}}^\mathfrak{t}.$$
	
	\begin{remark}\label{rem:notation}
		There is a bijection between vertices of $\Gamma_{\varphi}$ and two dimensional faces $\Delta_i$ of $\Delta_\varphi$. Both sets biject with irreducible components $X_i$ of $X_{\Gamma_{\varphi}}$. Each $\Delta_i$ is a lattice polytope in $M_X$ and thus specifies a toric variety and $(\mathbb{C}^\star)^2$ linearised line bundle. This toric variety is isomorphic to $X_i$. The $(\mathbb{C}^\star)^2$ linearised line bundles glue to give a linearised line bundle on $X_{\Gamma_{\varphi}}$. See \cite[Section~4]{aschermolcho} for background.
	\end{remark}
	
	Our next definition reflects language used in \cite{aschermolcho}.
	\begin{definition}
		A surface $X_{\Gamma_{\varphi}}$ equipped with a linearised line bundle $L$ arising as in Remark \ref{rem:notation} is called a \textit{broken toric surface}.
	\end{definition}

	We specify a $(\mathbb{C}^\star)^2$ linearisation of $\mathcal{O}(1)$ on $\mathbb{P}^m$. Recall the coordinates on $\mathbb{P}^m$ naturally biject with $\Delta$ and so can be labelled $X_{i,j}$. An element $(s,t)$ of $(\mathbb{C}^\star)^2$ acts on a basis of $\mathcal{O}_{\mathbb{P}^m}(1)$ by sending $$(s,t):X_{i,j}\mapsto s^it^jX_{i,j}.$$
	
	\begin{definition}
		Let $X_{\Gamma_\varphi}$ be a broken toric surface. A \textit{stable toric morphism to $\mathbb{P}^m$} is a morphism $$f:X_{\Gamma_\varphi}\rightarrow \mathbb{P}^m$$ and a $(\mathbb{C}^\star)^2$ equivariant isomorphism $f^\star(\mathcal{O}(1))\cong L$.
	\end{definition}
	\begin{definition}\label{def:STS}
		A curve $C^\dagger$ in $X_{\Gamma_{\varphi}}$ is \textit{strongly transverse and stable} if it arises by pulling back the section $\sum_{i,j} X_{i,j}$ of $\mathcal{O}(1)$ along a finite stable toric morphism $X_{\Gamma_{\varphi}}\rightarrow \mathbb{P}^m$. Two curves $C^\dagger_1,C^\dagger_2$ are considered isomorphic if there is an isomorphism $$\phi:X_{\Gamma_{\varphi}}\rightarrow X_{\Gamma_{\varphi}}$$ which sends each irreducible component of $X$ to itself and which sends $C^\dagger_1$ to $C^\dagger_2$.
		
		A curve $C$ in $X_{\Gamma}\rightarrow X$ is strongly transverse and stable if it is the strict transform along $q_\Gamma$ of a strongly transverse and stable curve. Two such curves $C,C'\rightarrow X_\Gamma \rightarrow X$ are considered isomorphic if we have the following commutative triangle with $g$ an isomorphism such that $g(C) = C'$.
		$$
		\begin{tikzcd}
		X_\Gamma \arrow[r, "g"] \arrow[d] & X_\Gamma \arrow[ld] \\
		X                                 &                    
		\end{tikzcd}  
		$$
	\end{definition} 
	
	\subsubsection{Stable curves and maps from broken toric surfaces.} We show that strongly transverse and stable curves are precisely those curves obtained by intersecting the subscheme $\mathcal{Z}$ defined in Section~\ref{sec:flattening} with $\varpi^{-1}(p)$ for some point $p$ of $\M$. Note strongly transverse and stable curves are subcurves of an expansion and have no relation to the stable curves studied in Gromov--Witten theory. Definition \ref{def:STS} is motivated by the following consequence of \cite[Theorem 1]{aschermolcho}.
	
	\begin{theorem}\label{thm:AscherMolcho}
		Consider the diagram used in Construction~\ref{cons:univsubsch}	$$
		\begin{tikzcd}
		\R \arrow[d,"\kappa"] \arrow[r,"\pi_{\mathbb{P}^m}"] & \mathbb{P}^m \\
		\VSp                &             
		\end{tikzcd}.
		$$ 
		For each stable toric morphism $g:X_{\Gamma_{\varphi}}\rightarrow \mathbb{P}^m$ arising from $L$ there is a unique closed point $x$ in $\VSp$ such that $$\pi_{\mathbb{P}^m}|_{(\varpi^{\dagger})^{-1}(x)}:\kappa^{-1}(x)\rightarrow \mathbb{P}^m$$ coincides with $g$.
	\end{theorem} 
	A strongly transverse and stable curve in a broken toric surface is a Cartier divisor. Fixing once and for all a generic section of $O(1)$ on $\mathbb{P}^1$, this observation means the above theorem induces a bijection between points of $\mathcal{P}_\Delta$ and strongly transverse and stable curves in broken toric surfaces.
	\begin{proof}
		Ascher and Molcho identify the coarse moduli space of logarithmic maps from stable toric varieties to $\mathbb{P}^m$ with the Chow quotient of $\mathbb{P}^m$ by a two dimensional torus. This Chow quotient coincides with $\VSp$ by \cite{QuotientsToricVar}. The coarse moduli space carries universal data: the universal family of toric surfaces is precisely $\R\rightarrow \VSp$ and there is a universal morphism $\R \rightarrow \mathbb{P}^m$. We have thus characterised the diagram in the statement of our theorem in terms of a coarse moduli space and universal data and our claim is immediate.
	\end{proof}
	
	The logarithmic boundary of a toric surface $X_\Gamma$ is the divisor $D$ consisting of points which do not lie in the dense torus of any irreducible component of $X_\Gamma$ . A curve $C$ on a broken toric surface is \textit{transverse} if the intersection of $C$ with the logarithmic boundary $D$ consists of finitely many points contained within the smooth locus of $D$. We apply Theorem \ref{thm:AscherMolcho} to check that strongly transverse and stable curves intersect the logarithmic boundary of $X_\Gamma$ or $X_{\Gamma_{\varphi}}$ transversely.  It suffices to check this property for curves on $X_{\Gamma_{\varphi}}$; indeed if $C$ intersects the logarithmic boundary of $X_{\Gamma_{\varphi}}$ transversely then $q_\Gamma^{-1}(C)$ intersects the logarithmic boundary of $X_\Gamma$ transversely. 
	\begin{lemma}\label{lem:strongtransstab}
		Let $C \rightarrow X_{\Gamma_{\varphi}}$ be a strongly transverse and stable curve. The curve $C$ intersects the boundary of $X_{\Gamma_{\varphi}}$ transversely.
	\end{lemma}
	
	The proof of Lemma~\ref{lem:strongtransstab} uses the \textit{tropicalisation of a Laurent polynomial} which we now recall. Write $R=\mathbb{C}[X_0^{\pm 1},...,X_n^{\pm 1}]$ and $T$ for the torus $\mathrm{Spec}(R)$ with cocharacter lattice $M_T$. The \textit{tropicalisation of $f$} \cite[Definition 3.1.1]{MaclaganSturmfels} is a subset $\mathrm{Trop}(f)$ of $M_T \otimes \mathbb{R}$. For $Y$ a proper toric variety with dense torus $T$ and fan $\Sigma_Y$, the closure of $V(f)$ in $Y$ intersects a torus orbit $O(\sigma)$ corresponding to $\sigma$ if and only if $\mathrm{Trop}(f)$ intersects the interior of $\sigma$ \cite[Theorem 6.3.4]{MaclaganSturmfels}.
	
	\begin{proof}
		We set $$x\in \Sp\quad X_x^\dagger=\kappa^{-1}(x)\quad {Z}_x^\dagger={Z}^\dagger \cap X_x.$$ By Theorem \ref{thm:AscherMolcho} there is a point $x$ of $\Sp$ such that ${Z}^\dagger_x \rightarrow X_x^\dagger$ is isomorphic to $C \rightarrow X_{\Gamma_{\varphi}}$ so we are left to check ${Z}_x^\dagger\hookrightarrow X_x^\dagger$ intersects the boundary transversely for all such $x$.
		
		We first prove ${Z}_x^\dagger$ does not intersect any strata of codimension two in the boundary of $X_x^\dagger$. Indeed the proof of Lemma \ref{lem:auxcomb} showed that a point $y$ in $X_x^\dagger$ lies in a stratum of $X_x^\dagger$ of codimension two if and only if the cone in $\TR$ corresponding to the torus orbit of $y$ is mapped to a maximal cone in the fan of $|\beta|^\mathfrak{t}$ under $\pi^\mathfrak{t}$. The tropicalisation of the Laurent polynomial defining the restriction of ${Z}^\dagger$ to the dense torus in $\mathbb{P}^m$ is the support of the cones of codimension one in $||^\mathfrak{t}$. By \cite[Theorem 6.3.4]{MaclaganSturmfels} the point $y$ cannot lie in the hyperplane $V\left(\sum_{(i,j)}X_{i,j}\right)$.
		
		We check ${Z}_x^\dagger$ does not contain any dimension one logarithmic stratum of $X_x^\dagger$. Restricting $\pi_{\mathbb{P}^m}$ to $X_x^\dagger$ is a torus equivariant map to $\mathbb{P}^m$. Consider the interior $O(\gamma)$ of an irreducible component of the logarithmic boundary of $X_x^\dagger$. Fix an isomorphism from $O(\gamma)$ to $\mathbb{C}^\star$ and an isomorphism from its closure $V(\gamma)$ to $\mathbb{P}^1$. We must show the image  of $V(\gamma)$ in $\mathbb{P}^m$ intersects $\sum_{i,j} X_{i,j}=0$ at finitely many points. Suppose the origin in the dense torus is mapped to $[b_{i,j}]$. Then since our map is torus equivariant a point $t$ of $\mathbb{C}^\star$ is taken to $[t^{i+j}b_{i,j}]$. This point lies in the zero set of $\sum_{i,j}X_{i,j}$ if and only if $$k(t)=\sum_{i,j}t^{i+j}b_{i,j}=0.$$ This is not the zero polynomial since $\pi_{\mathbb{P}^m}$ restricted to $X^\dagger_x$ is finite, so $k(t)$ has finitely many roots and the result is proved.
	\end{proof}

	\subsection{Torus action}\label{torusact}
	
	Note $X$ is toric and thus admits the action of its dense torus $X^o =\left(\mathbb{C}^\star\right)^2$. 
	
	\subsubsection{Torus action on strongly transverse and stable curves.} Denote the set of isomorphism classes of strongly transverse and stable curves $$\mathcal{\mathcal{STS}} =\{\textrm{isomorphism classes of strongly transverse and stable curves }{C}\hookrightarrow X_\Gamma\rightarrow X\}.$$ In this definition $X_\Gamma\rightarrow X$ is any expansion of $X$. An element $x$ of $X^o$ whose action on $X$ defines an automorphism $\phi_x: X\rightarrow X$ on $X$ induces an automorphism of the set $\mathcal{\mathcal{STS}}$ by sending $$C \hookrightarrow X_\Gamma \xrightarrow{\pi_X} X\textrm{ to } C \hookrightarrow X_\Gamma \xrightarrow{\phi_x \circ \pi_X} X.$$ In this way we have defined a group action of $X^o$ on $\mathcal{STS}$.
	
	\subsubsection{Subgroups of $X^o$} Fix a combinatorial type $\mathcal{C}(\Gamma)$ of pre--expansion tropical curve and a strictly transverse and stable curve $\iota^\dagger:C^\dagger \hookrightarrow X_{\Gamma_\varphi}$. In this subsection we associate a subgroup $X^o_\Gamma$ of $X^o$ to the pair $(\mathcal{C}(\Gamma),\iota^\dagger)$. Fix a choice of representative $\Gamma \in \mathcal{C}(\Gamma)$.
	
	\begin{definition}
		The \textit{extended main component} of a pre--expansion tropical curve $\Gamma$ is the collection of points $p$ in $\Gamma$ such that the convex hull of $\{0,p\}$ lies within $\Gamma$.
	\end{definition}
	
	The pre--expansion tropical curve $\Gamma$ arose by superimposing a tropical curve $\Gamma_\varphi$ on to the fan of $X$. In the next paragraph by a vertex of $\Gamma_\varphi$ inside the extended main component we always mean a point in the extended main component which is a vertex of $\Gamma_\varphi$. 
	
	The subcurve $\iota^\dagger$ defines a face of the secondary polytope, see the observation following Theorem~\ref{thm:AscherMolcho}. Faces of the secondary polytope specify a subdivision $\Delta_\varphi$ as well as the set $D_S$ of $\Delta$ on which $\varphi$ adopts the same value as $\hat{\varphi}$. For a vertex $v$ of $\Gamma_\varphi$ corresponding to face $\delta_v$ of $\Delta_\varphi$ we write $M_v$ for the subgroup of $M_X$ generated by elements of $M_X$ contained in $\delta_v$ on which $\varphi$ and $\hat{\varphi}$ coincide. 
	
	The intersect of the kernels of the cocharacters in $M_v$ is a subgroup of $X^o$. We define a subgroup $X^o_\Gamma$ of $X^o$ to be the intersect over vertices $v$ in $\Gamma_\varphi$ of $X_v^o$ where we define $X_v^o$ as follows. 
	
	If $v$ is at the origin then $X_v^o$ is the group corresponding to the sublattice $M_v$. If $X_v$ is not in the extended main component we set $X_v = X^o$. Every other vertex $v$ of $\Gamma_\varphi$ lies in the extended main component and so we can find a unique vector $\textbf{n}$ in the direction of a ray in the fan of $X$ such that $\varphi + \lambda \textbf{n}$ adopts its minimal value on the corners of the face $\delta_v$ of $\Delta_\varphi$ for some real $\lambda$. Write $M(v)$ for the minimal subgroup of $M$ whose image in $M/\textbf{n}$ coincides with the image of $M_v$. Define $X^o_v$ to be the intersect of the kernels of characters in $M(v)$.
	
	\begin{remark}\label{rem:torusact}
		The inclusion $N_X \rightarrow N$ defines an action of $X^o$ on the logarithmic linear system. Dualising to get a map on characters, the stabiliser of the torus orbit corresponding to $\Gamma$ is $X_\Gamma^o$.
	\end{remark}
	\begin{example}
		Let $\Delta$ be the convex hull of $(0,0), (2,0)$ and $(0,2)$ and consider the face associated to the map $\varphi$ sending $(0,0), (2,0),(0,2)\mapsto 0$ and other elements of $\Delta$ to positive numbers. Then $$X_\Gamma^o = \{(1,1),(-1,1),(1,-1), (-1,-1)\}\subset (\mathbb{C}^\star)^2.$$ Geometrically this face corresponds to transverse curves on the trivial expansion $X$ with equations of the form $$aX^2 + bY^2 + c \textrm{ where } a,b,c \textrm{ are in }\mathbb{C}^\star.$$ 
	\end{example}
	\subsubsection{Stabilisers} Fix $\Gamma$ a combinatorial type of pre--expansion tropical curve obtained by superimposing $\Gamma_\varphi$ on to the fan of $X$. Fix also a map $\iota:C^\dagger \hookrightarrow X_{\Gamma_{\varphi}}$. Denote $\mathcal{STS}_\Gamma(\iota)$ the subset of $\mathcal{STS}$ consisting of maps $C\hookrightarrow X_\Gamma \rightarrow X$ pulled back along $q_\Gamma$ from $\iota$.

	\begin{proposition}\label{prop:stabilisers}
		For any element $p$ of  $\mathcal{STS}_\Gamma(\iota)$, the subgroup of $X^o$ stabilising $p$ is $X^o_\Gamma$.
	\end{proposition}
	\begin{proof}
		Let $x$ be an element of $X^o_\Gamma$ and consider a diagram $$
		\begin{tikzcd}
		C \arrow[r, hook] & X_\Gamma \arrow[r] \arrow[d, "h", dashed] & X \arrow[d, "\varphi_x"] \\
		C \arrow[r, hook] & X_\Gamma \arrow[r]                        & X.                       
		\end{tikzcd}$$ We construct an isomorphism $h$ making the diagram commute and mapping $C$ to itself. It follows that $x$ acts trivially on $\mathcal{STS}_\Gamma(\iota).$
		
		We specify the map $h$ by specifying the restriction of $h$ to each component in a compatible way. The restriction of $h$ to the component mapped birationally to $X$ is specified by $\varphi_x$. The restriction of $h$ to components corresponding to vertices which are not in the extended main component is the identity. The restriction of $h$ to remaining components corresponding to vertices of $\Gamma_\varphi$ is characterised by $\varphi_x$ on divisors in the extended main component and the identity on other divisors.
		
		The components $Y_i$ corresponding to all remaining vertices are isomorphic to toric surfaces whose fans have four rays of the form $\pm \rho_1, \pm \rho_2$. After swapping labels we may assume $\rho_1$ points towards the origin. Thus the restriction of $h$ to $O(\rho_1)$ is specified by $\varphi_x$. The restriction of $h$ to $O(\pm \rho_2)$ is the identity. There is a unique torus equivariant automorphism of $Y_i$ which restricts in this way to the toric boundary. 
		
	\end{proof}
	
	\subsection{Points of $\VM$}
	In this section we show points of $\VM$ biject with strongly transverse and stable curves on expansions $X_\Gamma \rightarrow X$.  For $y$ in $\VM$ set $X_y=\varpi^{-1}(y)\rightarrow X$ the associated expansion. 
	
	\begin{theorem}\label{thm:univcurve}
		There is a natural bijection: $$\alpha: \mathrm{Hom}(\mathrm{Spec}(\mathbb{C}),\VM) \rightarrow \mathcal{STS},$$ where $\alpha$ assigns to a point $y$ the curve $${Z}_y={Z}\cap X_y\hookrightarrow X_y\rightarrow X.$$ 
	\end{theorem}
	We first check $\alpha$ is well defined with Lemma \ref{lem:fibreright}. 
	\begin{lemma}\label{lem:fibreright}
		The curve ${Z}_y={Z}\cap X_y\hookrightarrow X_y$ is strongly transverse and stable.
	\end{lemma}
	
	\begin{proof}[Proof of Lemma \ref{lem:fibreright}]
		In Construction \ref{cons:univsubsch} the subscheme ${Z}$ was defined to be the strict transform of ${Z}^\dagger$ under $\pi_{\mathbb{P}^m}$. Fibre--wise this implies ${Z}_y\rightarrow X_\Gamma \rightarrow X$ is the strict transform of ${Z}_y^\dagger\rightarrow X_{\Gamma_{\varphi}}$ under $q_\Gamma$. Lemma \ref{lem:strongtransstab} shows the latter curve is strongly transverse and stable.
	\end{proof}
	We will deduce Theorem \ref{thm:univcurve} from a similar result for $\mathrm{Hom}(\mathrm{Spec}(\mathbb{C}),\VSp)$ recorded in Lemma~\ref{lem:uniqueyrub}. Fix a strongly transverse and stable curve $C^\dagger \rightarrow X_{\Gamma_{\varphi}}$. 
	\begin{lemma}\label{lem:uniqueyrub}
		There is a unique $\mathbb{C}$ point $y^\dagger$ of $\VSp$ such that the strongly transverse and stable curve $$Z^\dagger\cap \kappa^{-1}(y^\dagger)=C_{y^{\dagger}} \hookrightarrow {X}_{y^\dagger} = \kappa^{-1}(y^\dagger)$$ coincides with $C^\dagger \hookrightarrow X_{\Gamma_{\varphi}}$.
	\end{lemma}
	\begin{proof}
		Theorem \ref{thm:AscherMolcho} shows such a point $y^{\dagger}$ exists. We check $y^\dagger$ is unique.
		
		Recall the surface $X_{\Gamma_\varphi}$ has irreducible components $X_i$. Set $C_i = C^\dagger \cap X_i$ and observe $C_i$ defines an ample line bundle $L_i$ on $X_i$. This line bundle specifies a lattice polytope up to translation. A subdivision of $\mathrm{supp}(\Delta)$ with dual tropical curve $\Gamma_{\varphi}$ assigns a line bundle to each vertex of $\Gamma_{\varphi}$. There is a unique subdivision of $\mathrm{supp}(\Delta)$ giving rise to $X_{\Gamma_{\varphi}}$ equipped with the line bundle $L_i$ on $X_i$ for each $i$. Such a subdivision specifies the cone $\sigma$ of $\TSp$. We have thus identified $\sigma$ for which $y^\dagger$ in $O(\sigma)$.
		
		Theorem \ref{thm:AscherMolcho} shows that points of $O(\sigma)$ biject with stable toric morphisms $X_{\Gamma_{\varphi}}\rightarrow \mathbb{P}^m$ where $\varphi$ lies in the interior of $\sigma$. We identify the restriction of the stable toric morphism corresponding to $y^\dagger$ to $X_i$. Chasing definitions we may write
		$$C_i = V\left(\sum_{(i,j) \in \Delta_i}\lambda_{i,j} x^iy^j\right).$$
		
		Since $C_i$ is the pullback of $\sum_{i,j}X_{i,j}$ under the morphism $f_{y^{\dagger}}$ necessarily $f_{y^{\dagger}}$ restricted to the dense torus of $X_i$ is
		$$ [f_{y^{\dagger}}(x,y)]_{k,l} = \begin{cases}
		\lambda_{k,l}x^ky^l \quad &(k,l) \in \Delta_i \\
		0 \quad \quad \quad \quad  &(k,l) \notin \Delta_i
		\end{cases}.
		$$
		In this way we have specified a morphism $f_{y^\dagger}$ corresponding to $C^\dagger$. Applying this construction to two isomorphic curves $C^\dagger_1\hookrightarrow X_{\Gamma_\varphi}$ and $C^\dagger_2\hookrightarrow X_{\Gamma_\varphi}$ arising from points $y_1^\dagger, y_2^\dagger$ of $\VSp$ yields two maps $f_{y^\dagger_1},f_{y^\dagger_2}$ which fit into a commutative triangle $$
		\begin{tikzcd}
		X_{\Gamma_\varphi} \arrow[rd, "f_{y^\dagger_2}"] \arrow[d] &              \\
		X_{\Gamma_\varphi} \arrow[r, "f_{y^\dagger_1}"']           & \mathbb{P}^m
		\end{tikzcd}$$ where the vertical morphism is an isomorphism. Thus $y^\dagger_1$ and $y^\dagger_2$ correspond to the same stable toric morphism to $\mathbb{P}^m$ and are the same point.
	\end{proof}
	
	To prove Theorem \ref{thm:univcurve} we write down the inverse to alpha. Thus given a strongly transverse and stable curve $$\iota_{y^\dagger}:C \hookrightarrow X_\Gamma \rightarrow X$$ we must specify the corresponding $\mathrm{Spec}(\mathbb{C})$ point $y$ of $\VM$. 
	
	\begin{proof}
		By definition there is a strongly transverse and stable curve $C^\dagger \hookrightarrow X_{\Gamma_{\varphi}}$ such that $C$ is pulled back along $X_\Gamma \rightarrow X_{\Gamma_\varphi}$. By Lemma \ref{lem:uniqueyrub} we know the image $y^\dagger$ of $y$ under $\nu$. Torus orbits in the preimage under $\nu$ of the torus orbit of $y^\dagger$ biject with strata of $\mathfrak{P}_{\Gamma_\varphi^\mathrm{aug}}$. The combinatorial type of $\Gamma$ specifies the stratum of $\mathfrak{P}_{\Gamma_\varphi^\mathrm{aug}}$ corresponding to the torus orbit in which $y$ lies.
		
		We have thus identified a single torus orbit $O(\sigma)$ in which $y$ can lie. More precisely we know $y$ lies in $W_y = \nu^{-1}(y^\dagger)\cap O(\sigma)$. The inclusion $N_X\hookrightarrow N$ determines an action of $X^o$ on $W_y$ with respect to which the map $$\alpha: \mathrm{Hom}(\mathrm{Spec}(\mathbb{C}), W_y )\rightarrow \mathcal{STS}_\Gamma(\iota_{y^\dagger})$$ is $X^o$ equivariant. The action on the target is transitive and the stabilisers coincide by Proposition~\ref{prop:stabilisers} and Remark~\ref{rem:torusact}. Thus the map is a bijection.
	\end{proof}
	
	\section{Tautological classes as Minkowski weights and first computations}\label{sec:computations}
	
	The operational Chow ring of a toric variety is presented in combinatorial terms in \cite{IntToricVar}. We use this dictionary to study tautological integrals on $\M$ in terms of the geometry of the the fan $|\Delta|$. The complexity in the fan $|\Delta|$ arises from the secondary fan from which it is built. 
	
	In Section \ref{sec:one} we constructed the family $\varpi:\UM_0 \rightarrow \M$ as a morphism of toric stacks and a corresponding morphism of coarse moduli spaces $\tilde{\varpi}:\VUM \rightarrow \VM$. Tautological insertions are defined in terms of $\varpi$. Our approach is to pull back intersections performed on the coarse moduli space: that is we study $\tilde{\varpi}$. We rely on general notions from toric geometry; see \cite{Fulton+2016} for background.
	
	An element $B$ in the Chow group $A_k(\VM)$ may be expressed  $$B=\sum_{\tau} \lambda_\tau [V(\tau)].$$ Here $V(\tau)$ is the closure of the torus orbit $O(\tau)$ associated to a cone $\tau$ in the fan $|\Delta|$ under the toric dictionary. Thus to understand the action of an insertion on any homology class it suffices to understand its action on $V(\tau)$ for all cones $\tau$ of $|\Delta|$.
	
	The operational Chow ring of a toric variety $A^\star(X)$ is isomorphic to the ring of Minkowski weights on the fan of $X$. An element of the Chow group $A_\star(X)$ defines an element of the operational Chow ring via the intersection product \cite{Fulton}. Observe $A_\star(X)$ is a module over $A^\star(X)$. To compute products of Minkowski weights, or their action on elements of the Chow group we apply a fan displacement rule \cite[Theorem 4.2]{IntToricVar}. The following definition is useful in discussing these techniques.
	\begin{definition}
		Let $\Sigma$ be a fan. Consider a triple $(\rho, \tau, v)$ where $v$ lies in the support $|\Sigma|$ and both $\rho$ and $\tau$ are cones of $\Sigma$. Such a triple is \textit{good} if $\rho  \cap (\tau+v) \ne \emptyset$.
	\end{definition}
	
	Recall the fan $\VTM$ is a subdivision of the fan $\TPN$ of the linear system $|\beta|$. The cone of $\TPN$ in which a function $\varphi$ in $N_\mathbb{R}$ lies is specified by which points of $\Delta$ the function $\varphi$ adopts its minimal value. We denote the cone in $\TPN$ of functions adopting their minimal value on $S \subset \Delta$ by $\delta_S$. We partition the cones of $\VTM$ according to which cone of $\TPN$ they lie within. The set of cones in $\VTM$ lying within $\delta_S$ is denoted $D_S$. The set $|D_S|\subset N_\mathbb{R}$ is the union of the supports of the cones in $D_S$.
	
	We will denote the cocharacter lattice in which $\VTUM$ sits as $$N_{\VUM} = N\times_{\mathscr{N}}N.$$
	
	\subsection{Point insertions}\label{sec:two:ptinsert}
	
	Consider the universal diagram $$
	\begin{tikzcd}
	{\mathcal{Z}} \arrow[r, hook] & \UM_0 \arrow[d, "\varpi"] \arrow[r, "\pi_X"] & X                        \\
	& \M \arrow[r, "\mathrm{ev}"']           & \mathrm{Hilb}^{\mathrm{log}}_\beta(\partial X)
	\end{tikzcd}.$$ Let $[\mathrm{pt}]$ be the class in $A^\star(X)$ defined by intersection with a generic point. A point insertion is an element $\tau_0([\mathrm{pt}])$ in $A^1(\M)$ defined to act on a class $B$ in the Chow group $A_\star(\M)$ as  $$\tau_0([\mathrm{pt}])(B)=\varpi_{\star}(c_1(\mathcal{J}_\mathcal{Z})\cap \pi_X^\star([pt])\cap \varpi^\star(B)).$$ Here $\mathcal{J}_\mathcal{Z}$ is the ideal sheaf of the universal subscheme $\mathcal{Z}$ inside $\mathcal{X}_0$. In this section we prove Proposition \ref{thm:identifyminkwt} identifying $\tau_0([pt])$ with the pullback of the hyperplane class on $\mathbb{P}^m$. We use this understanding to compute $\tau_0([pt])^k V(\tau)$ for cones $\tau$ of codimension $k$.
	
	Our approach is to work on the level of coarse moduli spaces. Let $\mathcal{J}_{Z}$ be the ideal sheaf of the universal subscheme ${Z}$ inside $\VUM$. We start by understanding the class $c_1(\mathcal{J}_{Z})\cup \pi_X^\star([pt])$ in $A^3(\VUM)$ as a Minkowski weight.
	
	\begin{definition}
		For $\sigma$ a cone of $\VTM$ we define a cone in the fan $X_{\mathrm{uni}}^\mathfrak{t}$ by $$\sigma_{eq} = \{(\varphi,\varphi)| \varphi \in |\sigma|\}.$$ Let $\hat{\sigma}$ be the subset of $\mathrm{star}(\sigma_{eq})$ consisting of cones with dimension $\mathrm{dim}(\sigma) +2$ whose support are contained in $\varpi^{-1}(\sigma)$.
		
		Define $n_{\{(ij),(kl)\}}$ the lattice length of the line segment with ends $(i,j)$ and $(k,l)$. For $\sigma$ a maximal cone of $D_{\{(i,j),(k,l)\}}$ define $n_\sigma = n_{\{(i,j),(k,l)\}}$
	\end{definition}
	
	\begin{lemma}\label{lem:cup}\hfill
		\begin{enumerate}[(1)]
			\item The Chern class $c_1(\mathcal{J}_{Z})$ corresponds to the following Minkowski weight on cones of codimension one
			$$c_1(\mathcal{J}_Z)(\sigma) = \begin{cases}
			n_\sigma \quad &\varpi^\mathfrak{t}(\sigma) \in D_{\{(i,j),(k,l)\}} \quad \{(i,j),(k,l)\} \subset \Delta \\
			0 \quad \quad  &\mathrm{else} 
			\end{cases}.$$
			\item The class $\pi_X^\star([pt])$ corresponds to the following Minkowski weight on cones of codimension two:
			$$c(\sigma)=\begin{cases}
			1 \quad \textrm{if }\sigma = \tau_{eq}\textrm{ for some cone }\tau \\
			0 \quad \mathrm{else}
			\end{cases}.
			$$
			\item The class $c_1(\mathcal{J}_Z) \cup \pi_X^\star([pt])$ is represented by the following Minkowski weight on cones of codimension three:
			$$(c_1(\mathcal{J}_Z) \cup \pi_X^\star([pt]))(\sigma) = \begin{cases}
			n_\sigma \quad &\textrm{if }\sigma= \tau_{eq} \textrm{ for } \tau \in D_{\{(i,j),(k,l)\}} \{(i,j),(k,l)\} \subset \Delta\\
			0 \quad &\mathrm{else} 
			\end{cases}.$$
		\end{enumerate}
	\end{lemma}
	
	\begin{proof}\hfill
		\begin{enumerate}[(1)]
			\item The class $c_1(\mathcal{J}_Z)$ is the operational Chow class associated to intersection with the subscheme $Z\subset X_\mathrm{uni}$. This subscheme $Z$ is codimension one and by Lemma \ref{lem:strongtransstab} intersects the boundary transversely. It suffices to determine the length of the zero dimensional subscheme of intersection of $Z$ with each boundary stratum of dimension one. The tropical cones corresponding to such strata have codimension one. 
			
			There are two flavours of dimension one strata. First suppose this dimension one stratum is taken to a point under $\varpi$. In this case we think of our stratum as a component of the boundary in a pre--expansion. It is possible this stratum was introduced in subdividing $\TR$ to obtain $|\Delta|$. In this case $Z$ does not intersect the corresponding boundary component and our weight adopts the value zero. If this possibility does not occur then the tropical cone assigned to the stratum is codimension one with support consisting of pairs $(f,g)$ with $g$ in $D_{\{(i,j),(k,l)\}}$ where the function $g$ adopts its minimal value twice. The subscheme $Z$ intersects such a stratum in a subscheme of length $n_{\{(i,j),(k,l)\}}$. 
			
			The second flavour of dimension one strata are those not taken to a point under $\varpi$. Intersecting such a strata with a fibre gives a codimension two stratum in a pre--expansion. The subscheme $Z$ intersects the boundary of each fibre in the expected dimension by Lemma \ref{lem:fibreright} so never intersects such strata.
			
			\item A Minkowski weight may be pulled back along a surjective morphism to a Minkowski weight by \cite[Proposition 3.7]{IntToricVar}. Observe $\pi_X^\mathfrak{t}(N_{X_\mathrm{uni}})=N_X$ and the Minkowski weight corresponding to the operational class $[\mathrm{pt}]$ adopts value $1$ on the zero cone and $0$ elsewhere.
			\item This is an application of \cite[Proposition 4.1 (b)]{IntToricVar} to compute the cup product. All choices of fan displacement vector give the same weight.
		\end{enumerate}
	\end{proof}
	A technical proposition is required to prove Proposition~\ref{thm:identifyminkwt}.
	
	\begin{proposition}\label{prop:cuprod} 
		Let $\sigma$ be a cone of dimension $k$ in $|\Delta|$ and $V(\sigma)$ the associated torus orbit closure in the coarse moduli space of the logarithmic linear system. Suppose $B$ in $A_{m-k}(\M)$ satisfies $$F_{\M\star}(B)= V(\sigma).$$ Let $v$ be a fixed generic point in $N_{\VUM}$ and $\tau$ a cone of $\VTM$. Define the following integers $$\lambda_\tau = \sum [N:N_{\rho}+N_{\tau}].$$ The sum is over good triples $(\rho_\mathrm{eq},\nu,v)$ where:
		
		\begin{itemize}
			\item  $\kappa$ lies in $\hat{\tau}$
			\item $\sigma_\mathrm{eq}$ is a face of both $\rho_\mathrm{eq}$ and $\kappa$ 
			\item $\rho$ a maximal cone of $D_{\{(i,j),(k,l)\}}$ for some $(i,j),(k,l)$ points of $\Delta$.
		\end{itemize}	
		With this setup one has the following equality in the Chow group of $\VM$
		$$F_{\M\star}(\tau_0([pt])B) = \sum_{\sigma\textrm{ facet of }\tau} \lambda_\tau [V(\tau)].$$
		
	\end{proposition}
	
	\begin{proof}
		Since $\varpi$ is flat, chasing definitions we may write: $$F_{\UM,\star}(\varpi^\star B) = \sum_\tau \frac{[N:N_\sigma]}{[N_{X_\mathrm{uni}}:N_\tau]} [V(\tau)].$$ The sum is over $\tau$ such that $\tau$ surjects to $\sigma$ and $\mathrm{dim}(\tau) = \mathrm{dim}(\sigma)$. Such $\tau$ are cones of dimension $k$ consisting of functions $(f,g)$ with $f$ in $\sigma$ and either $f=g$ or $f \ne g$ and $g$ achieves its minimum on three points. If $\tau$ a cone of the latter type then $(c_1(\mathcal{J}_Z) \cup \pi_X^\star([\mathrm{pt}]))\cap V(\tau)=0$; indeed the Minkowski weight in part (3) of Lemma \ref{lem:cup} is supported on cones where $f=g$.
		
		Note for some choice of $a_\kappa$ we may write $$(c_1(\mathcal{J}_{Z}) \cup \pi_X^\star([\mathrm{pt}]))\cap F_{\UM,\star}(\varpi^\star B) = \sum a_\kappa V(\kappa).$$ Further notice  $$\varpi_{\star} [V(\kappa)]= \begin{cases}
		V(\tau) \quad &\kappa \in \hat{\tau}\\
		0 \quad \quad &\mathrm{else.}
		\end{cases}$$
		For $\kappa$ in $\hat{\tau}$ we are left to show $\varpi_\star(a_\kappa V(\kappa))= \sum_{\kappa \textrm{ fixed}} [N:N_{\rho}+N_{\tau}]$. 
		
		For $\kappa$ in $\hat{\tau}$ we compute $a_\kappa$ with \cite[Proposition 4.1]{IntToricVar}. We sum over the contribution of the cones $\rho_{eq}$ on which the Minkowksi weight in part (3) of Lemma \ref{lem:cup} is supported. Having fixed $v$ in $N_{\VUM}$ a contribution to $a_\kappa$ of $n_{\rho_{eq}}[N_{\VUM}:N_{\rho_{eq}}+N_{\kappa}]$ arises from $\rho_{eq}$ when $(\rho_{eq},\kappa,v)$ is a good triple and $\sigma_{eq}$ is a face of both $\rho_{eq}$ and $\kappa$. For such cones $$\frac{[N_{\VUM}:N_\tau]}{[N:N_\sigma]} = n_\rho.$$ Finally observe $[N_{\VUM}:N_{\rho_{eq}}+N_{\kappa}] = [N:N_{\rho}+N_{\tau}]$.
		
		By the projection formula we obtain
		$$F_{\UM,\star}\left((c_1(\mathcal{J}_Z) \cup \pi_\mathcal{X}^\star([\mathrm{pt}]))\cap \varpi^\star B\right) = \sum a_\kappa V(\kappa).$$
		Pushing forward we obtain
		$$\tilde{\varpi}_{\star}F_{\UM,\star}\left((c_1(\mathcal{J}_\mathcal{Z}) \cup \pi_\mathcal{X}^\star([\mathrm{pt}]))\cap \varpi^\star B\right)=\sum_{\sigma\textrm{ facet of }\tau} \lambda_\tau [V(\tau)].$$
		Since the square in Section \ref{sec:flattening} commutes we are done.
	\end{proof}

	Let $[H]$ in $A^1(\mathbb{P}^m)$ be the hyperplane class. It is easy to see that the following Minkowski weight on cones of codimension one corresponds to the class $\pi_{\mathbb{P}^m}^\star([H])$:
	$$
	c(\tau) = \begin{cases}
	1 \quad &\tau\in D_{\{(i,j),(k,l)\}} \textrm{ maximal, for some } (i,j), (k,l) \in \Delta\\
	0 \quad &\mathrm{else}. 
	\end{cases}
	$$
	\begin{theorem}\label{thm:auxtominkwt}
		Suppose an element $B$ of $A_\star(\M)$ pushes forward under $F_{\M}$ to $[V(\sigma)]$. Then we have the following equality in the Chow group $A_\star(\VM)$
		$$F_{\M\star}(\tau_0([\mathrm{pt}])(B)) = c \cap [V(\sigma)].$$ 
	\end{theorem}
	\noindent This is equivalent to Theorem \ref{thm:identifyminkwt} since the rational Chow group of $\VM$ is naturally identified with the rational Chow group of $\M$.
	\begin{proof}
		Applying the fan displacement rule with some choice of vector $u$ to perform the right hand intersection we have
		$$c \cap V(\sigma) = \sum_{\sigma\textrm{ facet of }\tau} \alpha_\tau V(\tau) .$$
		It suffices to show that for some permissable choices of $u$ and $v$ we obtain $\alpha_\tau = \lambda_\tau$. Recall both $u$ and $v$ may be chosen freely in the compliment of finitely many linear spaces. Thus we may fix $v = (\varphi,\phi)$ and $u=\varphi$. 
		
		Let $S_\tau$ be the set of good triples $(\rho_{eq}, \kappa,v)$ with respect to the fan $X_\mathrm{uni}^\mathfrak{t}$ satisfying the hypotheses of Proposition \ref{prop:cuprod}. Let $R_\tau$ be the set of good triples $(\rho, \tau, u)$ with respect to the fan $|\Delta|$. These sets biject naturally $$S_\tau \rightarrow R_\tau$$ $$(\rho_{eq}, \kappa,v) \mapsto  (\rho, \tau, u).$$ The inverse is specified by taking $\kappa$ as the unique cone in $\hat{\tau}$ intersecting $\rho_{eq}-(\varphi,\phi)$. 
	\end{proof}

	\begin{proposition}\label{prop:paths}
		Let $\sigma$ be a cone of codimension $k$, let $B$ in $A^\star(\M)$ and suppose $$F_{\M\star}(B) = [V(\sigma)].$$  
		Letting $p$ be a reduced point in $\M$,
		$$F_{\M\star}(\tau_0([\mathrm{pt}])^kB) = \begin{cases}
		[\mathrm{p}]\quad \quad  &\sigma \in S_\delta ,|\delta|=k\\
		0 \quad \quad \quad \quad \quad &\mathrm{else}.
		\end{cases}
		$$
	\end{proposition}
	\begin{proof}
		In light of Theorem \ref{thm:auxtominkwt} we must compute $c^k\cap V(\sigma)$. Note $c$ corresponds to the pull back of $[H]$ under $\pi_{\mathbb{P}^m}$. Thus $c^k$ is the pullback of $[H]^k$. Pullback of Minkowski weights in this situation is described in \cite[Proposition 3.7]{IntToricVar}. 
	\end{proof}
	
	\subsection{Logarithmic insertions}
	To impose logarithmic insertions we pull back classes in the toric variety $\mathrm{Hilb}^{\mathrm{log}}_\beta(\partial X)$ along the toric morphism $\mathrm{ev}$. Pulling Minkowski weights back along morphisms which are either surjective or flat is straightforward, however $\mathrm{ev}$ is neither. Set $\pi_{i}$ the projection from $\mathrm{Hilb}^{\mathrm{log}}_\beta(\partial X)$ to the logarithmic Hilbert scheme of points on the $i^{th}$ component of the boundary $\VB_{k_i}$.
	
	\begin{lemma}
		The composition $\pi_i \circ \mathrm{ev}$ is surjective.
	\end{lemma}
	\begin{proof}
		This map is between two complete varieties so we need only check surjectivity to a dense set. Restricting to the map of tori the result is clear.
	\end{proof}
	
	\FloatBarrier
	\begin{remark}
		
		The morphism $\pi_i \circ \mathrm{ev}$ is not always flat. Indeed consider $\Delta = \mathrm{conv}\{(0,0),(1,1),(1,-1)\}$. Figure \ref{fig:counterex} depicts a combinatorial type of pre--expansion tropical curve corresponding to a cone of $|\Delta|$. 
		
		\begin{figure}
			\includegraphics[width=100mm]{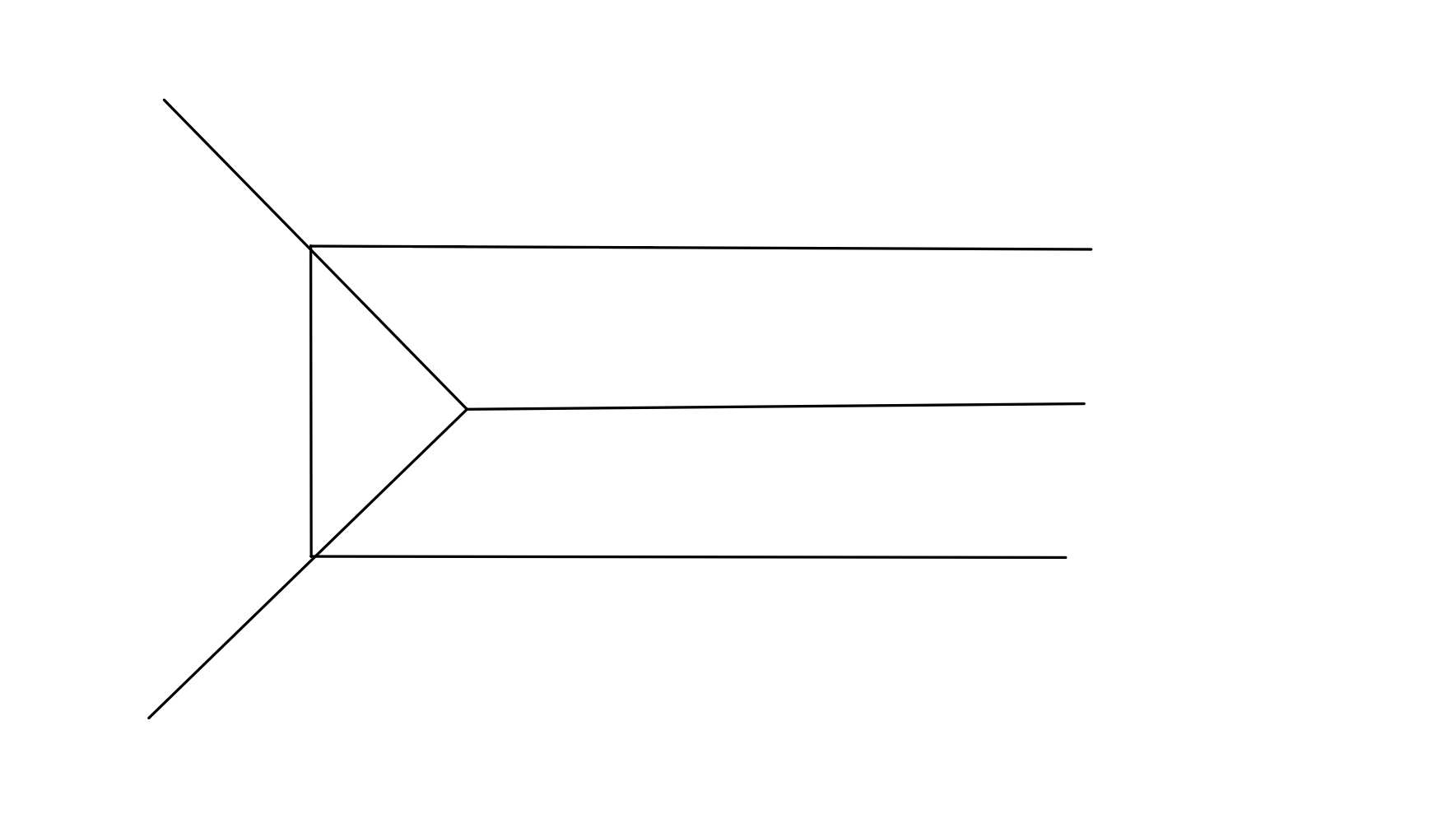}
			\caption{An example of a cone whose image under $\pi_i^{\mathfrak{t}} \circ \mathrm{ev}^{\mathfrak{t}}$ is not a whole cone. Via the toric dictionary we deduce $\pi_i \circ \mathrm{ev}$ is not flat.}
			\label{fig:counterex}
		\end{figure}
	\end{remark}
	
	One can easily pull back Minkowski weights along surjective morphisms \cite[Proposition~3.7]{IntToricVar}. In general to pull back any element of $A^\star(\mathrm{Hilb}^{\mathrm{log}}_\beta(\partial X))$ use K\"{u}nneth to write it as a linear combination of classes $x=\bigcup_i \pi_i^\star(\delta_i)$ for $\delta_i$ in $A^k(P_{k_i})$. Then $\mathrm{ev}^\star(x) = \bigcup_i \mathrm{ev}^\star\pi_i^\star(\delta_i)$.
	
	\begin{example}\label{ex:singlept}
		Let $B$ be the substack of $\mathcal{P}_d$ parametrising subschemes of length $d$ supported on a single point. Let $a$ be the class in the Chow ring $A^\star(\mathcal{P}_d)$ defined by intersection with $B$. We understand the Minkowski weight on $\VTB_{d}$ whose corresponding operational Chow class on $\VB_{d}$ pulls back under $F_{\B_d}$ to $a$. 
		
		The cycle $F_{\B_d\star}(B)$ is the closure of the collection of ideal sheaves supported in a single point in the dense torus of $\mathbb{P}^1$. Subschemes in this dense torus are specified by polynomials $\sum_{i=0}^na_iX^i$. Such a polynomial has a single root if and only if the following relations are satisfied for $k$ between $1$ and $d-1$ $$f_k=\binom{d}{k}\left(\frac{a_{d-1}}{d}\right)^k - a_{d-k}.$$ The tropicalisation (in the sense of \cite{MaclaganSturmfels}) of $V(f_1,...,f_d)$ is the intersection of the tropicalisations of $V(f_i)$. The tropicalisation of $f_i$ is the collection of points $$[X_0:...:X_d]\in \mathbb{R}^{k+1}/\mathbb{R}$$ where $iX_i = X_0$. Imposing a polyhedral structure, each maximal cone has weight $i$. The intersection of these tropicalisations is a line. The length of the intersection of the closure of $V(f_2,...,f_d)$ with the torus orbits corresponding to the two rays contained within this line is one. Indeed $f_k=0$ expresses $a_{d-k}$ in terms of $a_{d-1}$ unambiguously. This line is the union of two rays in the fan $\VTB_d$ so by \cite[Proposition 6.3.4]{MaclaganSturmfels} the associated Minkowski weight adopts the value one on these rays and zero elsewhere by the above analysis.
	\end{example}
	\subsection{A first computation with the logarithmic linear system}\label{sec:firstcomp}
	
	Computation \ref{comp:curvesinP1} is performed. A back of the envelope calculation suggests $(d!)^3$ is correct, and not $1$. The insertion of $(d-4)(d-2)/2$ point constraints means our task is to compute the degree of the codimension ${(d-4)(d-2)}/{2}$ subvariety $B$ studied in Example \ref{ex:singlept}. We are thus computing the degree of the intersection of $V(\{f| f \in I\})$ for $I$ a set of $3d$ polynomials, three of degree $i$ for $i=1,...,d$. By intersection theory on projective space the intersection of cycles cut out by such polynomials is degree $(d!)^3$. The particular form of the constraint polynomials, together with the geometry of the logarithmic linear system will force the answer one instead. The answer $1$ we obtain should be expected in light of Example \ref{ex:singlept}.
	
	For the remainder of this subsection fix the following polytope $$\Delta = \mathrm{conv}\{(0,0), (d,0), (0,d)\}.$$ This polytope has three faces of dimension one: Face 1 is the line segment $[(0,0), (0,d)]$ Face 2 is $[(0,0), (d,0)]$ and Face 3 is $[(0,d),(d,0)]$. Let $A_i$ be the collection of cones in $|\Delta|$ consisting of functions $\varphi$ which are linear when restricted to Face $i$.
	
	Let $a$ be the class in the Chow group $A_\star(\B_d)$ consisting of length $d$ ideal sheaves on $\mathbb{P}^1$ supported in a single point. The Minkowski weight of this subscheme is described in Example \ref{ex:singlept}. Set $a_i = \pi_i^\star(a)$.
	
	\begin{lemma}\label{lem:cupcupcup}
		There is an equality of Chow classes $$F_{\mathcal{PT}_0(X,\beta)\star}(\mathrm{ev}^\star(a_1\cup a_2 \cup a_3)\cap [\M]) = V(\tau)$$ where $\tau$ is the unique maximal cone of $D_S$ where $S$ is defined by $$S=\{(i,j) \in \Delta| i,j\ne0, i+j \ne d\}.$$
	\end{lemma}
	
	\begin{proof}
		The class $\mathrm{ev}^\star(a_i)$ is pulled back from the Minkowski weight which adopts the value one on the codimension $d-1$ cones in $A_i$. The Minkowski weight whose class pulls back to $\mathrm{ev}^\star(a_1)\cup \mathrm{ev}^\star(a_2) \cup \mathrm{ev}^\star(a_3)$ thus has support in $A_1\cap A_2\cap A_3$. A straightforward computation (with any generic choice of displacement vector) shows this Minkowski weight adopts the value one on all such cones.
		
		We compute $\mathrm{ev}^\star(a_1\cup a_2 \cup a_3)\cap V(0)$ by applying \cite[Proposition 4.1]{IntToricVar} using displacement vector $v$ such that $\Delta_v$ is the trivial subdivision. 
	\end{proof}

	The result of Computation \ref{comp:curvesinP1} follows applying Lemma \ref{lem:cupcupcup}, Proposition \ref{prop:paths} and noting $\int$ means push--forward to a point. This pushforward factors through $F_{\M}$. 
	
	\section{Stable pairs spaces}\label{sec:HighereulerChar}
	The moduli space of stable pairs on the toric surface $X$ with discrete data $\beta$ and $n$ is denoted ${\mathrm{PT}}^{\mathrm{abs}}_n(X,\beta)$. We refer to this as the \textit{absolute} stable pairs space to emphasise the distinction with the logarithmic stable pairs space. This moduli space parameterises elements in the derived category of sheaves on $X$ of the form $[s: \mathcal{O}_X\rightarrow F]$. Here $F$ is a pure sheaf and $\mathrm{coker}(s)$ has dimension zero and length $n$. The kernel of $s$ is the ideal sheaf of a curve of class $\beta$. Observe $\mathrm{PT}^{\mathrm{abs}}_0(X,\beta)$ coincides with the linear system $\mathbb{P}^m$.
	
	Consider the universal family in the absolute $n=0$ case \begin{equation}\label{eq:ktcons}
	\begin{tikzcd}
	\textbf{C} \arrow[r] \arrow[rd] & X \times \mathbb{P}^m \arrow[d, "\pi"] \\
	& \mathbb{P}^m                                       
	\end{tikzcd}.\end{equation} Pandharipande and Thomas identified an isomorphism \cite[Proposition B.8]{PT2} $$\mathrm{PT}^{\mathrm{abs}}_n(X,\beta)\cong \mathrm{Hilb}^{[n]}(\textbf{C}\rightarrow \mathbb{P}^m).$$ Kool and Thomas expressed this relative Hilbert scheme of points as the zero set of a section of a vector bundle inside $\mathrm{Hilb}^{[n]}(X\times \mathbb{P}^m\rightarrow \mathbb{P}^m)$ \cite[Appendix A]{Kool_2014}.
	
	In this section we apply logarithmic analogues of these techniques to construct a space $\PTnXbeta$. Points of $\PTnXbeta$ specify stable pairs $[s: \mathcal{O}_{X_\Gamma}\rightarrow F]$ on \textit{n-weighted expansions} $X_{\hat{\Gamma}} \rightarrow X$.
	
	\subsection{The logarithmic version of Diagram \ref{eq:ktcons}}
	
	We first obtain the logarithmic analogue $\pi_n$ of the morphism $\pi$ in Diagram \ref{eq:ktcons}. Since the permitted class of expansions depends on $n$, the diagram will also depend on $n$.
	
	\begin{definition}
		Let $\sigma$ be a cone in the fan of $\mathcal{X}^\mathfrak{t}$. Define $k$ satisfying the following expression $$\mathrm{dim}(\varpi^\mathfrak{t}(\sigma)) +k = \mathrm{dim}(\sigma).$$  
		If $k=0$ we say $\sigma$ is a \textit{component cone}; if $k=1$ then $\sigma$ is a \textit{curve cone} and if $k=2$ then $\sigma$ is a \textit{point cone}.
	\end{definition}
	
	The logarithmic version of the morphism $\pi$ is not the relative Hilbert scheme of points of the morphism $$\mathcal{Z} \rightarrow \M$$ because in this relative Hilbert scheme points are able to fall into torus orbits corresponding to curve cones. This violates the strong transversality condition of \cite{MR20}. Instead we construct a morphism of algebraic stacks $${\pi}_{n}:{\mathpzc{X}}_n^o \rightarrow {\mathpzc{P}}_{n}.$$ Locally ${\mathpzc{P}}_{n}$ contains an open subset of $\M$ as a dense open. The extra points are stacky and their fibres in ${\mathpzc{X}}_{n}^o$ look somewhat like an expansion, except they have extra components. These extra components allow properness without violating strong transversality.
	
	\subsubsection{Outline of Construction} To construct the morphism $\pi_n$ we turn first to tropical geometry and study something like a tropical version of a relative Hilbert scheme of points. We break our construction into a series of steps. Before proceeding with the construction we explain the goal of each step.
	
	\noindent \textbf{Step 1.} Define a fan $(\mathpzc{P}_n^\mathfrak{t})''$ whose points specify tuples $(\Gamma, p_1,...,p_n)$ where $\Gamma$ is a pre--expansion tropical curve and $p_i$ are (labelled) points on $N_X \otimes \mathbb{R}$. Both the combinatorial type of $\Gamma$ and the stratum of $\mathfrak{P}_\Gamma$ in which each point $p_i$ lies are constant on cones. 
	
	\noindent \textbf{Step 2.} Obtain $(\mathpzc{P}_n^\mathfrak{t})'$ by removing the locus in  $(\mathpzc{P}_n^\mathfrak{t})''$ corresponding to tuples $(\Gamma,p_1,...,p_n)$ where not all $p_i$ lie within the one skeleton of $\mathfrak{P}_\Gamma$. This is essential if we hope to impose a canonical polyhedral structure on our final space: the moduli of points in a two dimensional polyhedron is not naturally a cone complex, and one would be forced to make a choice of cone complex structure. 
	
	\noindent \textbf{Step 3.} Subdivide $(\mathpzc{P}_n^\mathfrak{t})'$ to ensure we can construct a combinatorially flat universal family. Because we threw away problematic cones in Step (2), this subdivision is canonical. We are subdividing along the locus where the $p_i$ coincide.
	
	\noindent This concludes the tropical part of our construction. We now explain the algebraic steps.
	
	\noindent \textbf{Step 4.} The toric dictionary applied to the output of Step 3 yields a toric variety $\vardbtilde{\mathpzc{P}}_{n}$. The dimension of this toric variety is $2n$ greater than the dimension of $\M$ so cannot locally contain $\M$ as an open subscheme. The problem is that $\vardbtilde{\mathpzc{P}}_{n}$ tracks the location of $n$ points on the universal expansion. This data should not be recorded by $\mathpzc{P}_n$. There is an action of $(\mathbb{C}^\star)^{2n}$ on $\vardbtilde{\mathpzc{P}}_{n}$. Taking the quotient of $\vardbtilde{\mathpzc{P}}_{n}$ by this torus action forgets the location of the $n$ points but still remembers the expansion. Denote the quotient stack $\tilde{\mathpzc{P}}_{n}$.
	
	\noindent \textbf{Step 5.} The stack $\tilde{\mathpzc{P}}_{n}$ tracks the labels of each $p_i$. To forget this data we quotient by the action of the symmetric group on $n$ letters yielding a stack $\mathpzc{P}_n$.
	
	\noindent \textbf{Step 6.} There is a universal expansion $$\mathpzc{X}_n \rightarrow \mathpzc{P}_n.$$ We remove toric strata from $\mathpzc{X}_n$ so that when taking a relative Hilbert scheme of points we do not encounter points falling into the logarithmic boundary.
	
	\subsubsection{Detailed Construction} We now detail how to perform each step in our construction. 
	\begin{construction}\label{cons:N}
		(Constructing the logarithmic version of morphism $\pi_2$.) 
		\begin{enumerate}[(1)]
			\item Consider the following fibre product in the category of fans. The fibre product is taken over $|\Delta|$
			$$ (\mathpzc{P}_n^\mathfrak{t})'' = 
			{\prod_{i=1}^n}_{\VTM} X_{\mathrm{uni}}^\mathfrak{t}.
			$$
			and denote the projection map
			$$(\varpi_{n}'')^\mathfrak{t}: (\mathpzc{X}_n^\mathfrak{t})''=X_{\mathrm{uni}}^\mathfrak{t} \times_{\VTM} (\mathpzc{P}_n^\mathfrak{t})'' \rightarrow (\mathpzc{P}_n^\mathfrak{t})''.$$
			\item Note a point of the ambient lattice in which $(\mathpzc{P}_n^\mathfrak{t})''$sits is specified by $n$ ordered pairs of functions in $N_\mathbb{R}$ $$((f_1,g),...,(f_n,g)).$$ Remove cones $\sigma$ of $(\mathpzc{P}_n^\mathfrak{t})''$ containing a point $((f_1,g),...,(f_n,g))$ such that $f_i$ is in a point cone for some $i$.
			Remove the preimages of these cones in $(\mathpzc{X}_n^\mathfrak{t})''$. There is a map between the resulting spaces $$(\pi_{n}^\mathfrak{t})':(\mathpzc{X}_n^\mathfrak{t})' \rightarrow (\mathpzc{P}_n^\mathfrak{t})'.$$
			\item Subdivide those cones $\sigma$ in the fan $(\mathpzc{P}_n^\mathfrak{t})'$ containing points such that $f_i$ and $f_j$ lie in the same curve cone for some $i\ne j$. We subdivide along the locus where $f_i=f_j$ and pull back the subdivision along $(\pi_{n}^\mathfrak{t})'$ to obtain the following diagram 
			$$\pi_{n}^\mathfrak{t}:\mathpzc{X}_n^\mathfrak{t} \rightarrow \mathpzc{P}_n^\mathfrak{t}.$$
			\item Applying weak semistable reduction \cite{molcho2019universal} to the associated morphism of toric varieties, we obtain the following morphism of toric stacks
			$$\vardbtilde{\pi}_{n}:\vardbtilde{\mathpzc{X}}_n \rightarrow \vardbtilde{\mathpzc{P}}_{n}.$$
			There is a surjective toric morphism $$\vardbtilde{\mathpzc{P}}_{n}\rightarrow \M.$$ The kernel of the associated morphism of dense tori is a torus of dimension $2n$ acting on $\vardbtilde{\mathpzc{P}}_{n}$. Choose a lift of this action to $\vardbtilde{\mathpzc{X}}_n$. The morphism $\vardbtilde{\pi}_{n}$ descends to a morphism of quotient stacks of relative dimension two
			$$\tilde{\pi}_{n}:\tilde{\mathpzc{X}}_n \rightarrow \tilde{\mathpzc{P}}_{n}.$$
			\item 
			The symmetric group on n letters acts equivariantly on the domain and codomain of $\tilde{\pi}_{n}$. This action is induced by permuting terms in the product $(\mathpzc{P}_n^\mathfrak{t})''$ and respects the fan structure of $(\mathpzc{P}_n^\mathfrak{t})'$. Pass to the quotient by this group action
			$${\pi}_{n}:{\mathpzc{X}}_n \rightarrow {\mathpzc{P}}_{n}.$$
			\item A cone $\sigma$ of $\mathpzc{X}_n^\mathfrak{t}$ is called \textit{stable }if for every point $((f_1,g),...,(f_n,g))$ in $|\sigma|$ each $f_i$ lies in a component cone of $X_n$. Remove all cones of ${\mathpzc{X}}_n^\mathfrak{t}$ which are not stable. Through the toric dictionary we obtain an algebraic stack denoted $\vardbtilde{\mathpzc{X}}_{n}^o$. Mimicking the constructions in (4) and (5) we obtain the following diagram
			$${\pi}_{n}:{\mathpzc{X}}_{n}^o \rightarrow {\mathpzc{P}}_{n}.$$
		\end{enumerate}
	\end{construction}
	
	\begin{remark}\label{rem:divlogstruc}
		Both $\vardbtilde{\mathpzc{X}}_n$ and $ \vardbtilde{\mathpzc{P}}_{n}$ are toric stacks so carry a divisorial logarithmic structure induced by their toric boundary. This induces a divisorial logarithmic structure on ${\mathpzc{X}}_n$ and ${\mathpzc{P}}_{n}$.
	\end{remark}
	
	\begin{construction}\label{cons:Znprime}
		(Construction of $\mathcal{Z}_n'$ -- the logarithmic analogue of $\textbf{C}$ in Diagram \ref{eq:ktcons}) The projection map  $$p_1^\mathfrak{t}:(\mathpzc{X}_n^\mathfrak{t})''=X_\mathrm{uni}^\mathfrak{t} \times_{\VTM} (\mathpzc{P}_n^\mathfrak{t})'' \rightarrow X_\mathrm{uni}^\mathfrak{t}$$ defines a morphism of toric stacks $$p_1:\mathpzc{X}_n''\rightarrow \mathrm{X}_{\mathrm{uni}}.$$ Take the strict transform of ${Z}$, as defined in Construction \ref{cons:univsubsch}, under $p_1$ to give a subvariety of the toric stack corresponding to $(\mathpzc{X}_n^\mathfrak{t})''$. Taking strict transforms and quotients at every stage of Construction \ref{cons:N} defines $\mathcal{Z}_n'$ fitting into the following diagram. 
		
		\begin{equation*}
		\begin{tikzcd}
		\mathcal{Z}_n' \arrow[r]\arrow[rd] & \mathpzc{X}_n \arrow[d, "\pi_{n}"] \\
		& \mathpzc{P}_n                                     
		\end{tikzcd}.\end{equation*}
		We define $\mathcal{Z}_n'^o=\mathcal{Z}_n' \cap {\mathpzc{X}}_{n}^o$.
	\end{construction}
	\subsection{Construction of $\PTnXbeta$} To complete our construction we need to set up a condition generalising strongly transverse and stable subcurves. We will refer to labelled points in a tropical curve as \textit{marked points}. Marked points are not unbounded ends of a tropical curve, contrasting with the terminology used elsewhere. 
	\begin{definition}An \textit{n--weighted pre--expansion tropical curve }is a triple $(\Gamma,\{p_1,...,p_k\},w)$ where $\Gamma$ is a pre--expansion tropical curve equipped with marked points $\{p_1,...,p_k\}$. To each marked point $p_i$ we assign an integer $w(p_i)$ such that $w(p_i)$ is greater than zero for all $i$ and $\sum_{i=1}^k w(p_i) = n$.
	\end{definition}
	
	In the sequel we shorten the triple $(\Gamma,\{p_1,...,p_k\},w)$ to ${\Gamma_k}$. 
	
	\begin{definition}\label{defn:weighttrop}
		An \textit{n--weighted pre--expansion tropical curve} ${\Gamma_k} = (\Gamma,\{p_1,...,p_k\}, w)$ defines an n--weighted expansion $X_{{\Gamma_k}}$ as follows. For each $p_i$ not on a vertex of $\Gamma$ insert a two valent vertex at $p_i$. After repeating this for each $p_i$ we obtain the tropical curve $\Gamma'$. Through Construction \ref{cons:expansions} the tropical curve $\Gamma'$ gives rise to an expansion $X_{{\Gamma_k}}\rightarrow X$. The bijection from Remark \ref{rem:notation} furnishes a function $$w:\{\textrm{two dimensional components of }X_{{\Gamma_k}}\}\rightarrow \mathbb{N}\cup \{0\}.$$ Here the image of any component not containing a marked point is zero.
	\end{definition}
	
	\noindent Note the divisorial log structure in Remark \ref{rem:divlogstruc} defines a stratification of $\mathpzc{P}_n$. Strata specify n--weighted tropical curves.

	\begin{remark}
		Let $p$ a point of $\mathpzc{P}_n$ and consider $\pi_n^{-1}(p)$. This is isomorphic to $\vardbtilde{\pi}_{n}^{-1}(x)$ for a point $x$ in $\vardbtilde{\mathpzc{P}}_{n}$; suppose $x$ is a point in $O(\sigma)$ for $\sigma$ a cone of $\vardbtilde{\mathpzc{P}}_{n}^\mathfrak{t}$ and note $X_{{\Gamma_k}}\cong \pi_n^{-1}(x)$. The function $w$ from irreducible components of $\pi_n^{-1}(x)$ to $\mathbb{N}$ is specified as follows. There is a bijection between irreducible components of $X_{{\Gamma_k}}$ and vertices of $\Gamma'$. Choose a point $((f_1,g)...,(f_n,g))$ in $\sigma$. Each $f_i$ specifies a point on $\Gamma$ equipping us with the list of points $q_1,...,q_n$. A list of distinct points is $p_1,...,p_k$, these are all vertices of $\Gamma'$ by construction. For $v$ a vertex of $\Gamma'$ define $w(v)$ to be the number of times $v$ appears in the list $q_1,...,q_n$.
	\end{remark}
	
	\begin{definition}
		A stable pair $[s:\mathcal{O}_{X_{{\Gamma_k}}}\rightarrow F]$ is \textit{stable and strongly transverse} if and only if the following conditions are satisfied:
		\begin{enumerate}[(1)]
			\item Let $v$ be a two dimensional component of $X_{{\Gamma_k}}$. Length $w(v)$ of $\mathrm{coker}(s)$ is supported on component $v$. 
			\item The curve $\mathrm{ker}(s)$ is a strongly transverse and stable sub--curve in the sense of Definition~\ref{def:STS}.
		\end{enumerate}
	\end{definition}
	\begin{lemma}\label{lem:stablepairhilb}
		Let $x$ be a closed point of the relative Hilbert scheme $\mathrm{Hilb}^{[n]}(\mathcal{Z}_n^o/\mathpzc{P}_n)$. This gives the data of a zero dimensional scheme embedded in a curve $Z \rightarrow C$. Suppose $Z$ has ideal $\mathcal{J}_Z$. The following morphism associated to $x$ is a stable pair
		$$\mathcal{O}_C \rightarrow \mathcal{J}_Z^\star.$$
		
	\end{lemma}
	\begin{proof}
		See the proof of \cite[Proposition B.5]{PT}.
	\end{proof}
	
	Choose a closed point of $\mathrm{Hilb}^{[n]}(\mathcal{Z}_n'^o/\mathpzc{P}_n)$: this specifies a subscheme $Y$ of a curve $C$ on an expansion $X_{{\Gamma_k}}$ such that no component of $Y$ lies in the logarithmic boundary of $C$. We obtain a function $w'$ from the irreducible components of $X_{{\Gamma_k}}$ to $\mathbb{N}$ defined by assigning the length of $v \cap Y$ to component $v$.
	
	\begin{lemma}\label{lem:stabconditionopen}
		To each point $x$ of $\mathrm{Hilb}^{[n]}(\mathcal{Z}_n^o/\mathpzc{P}_n)$ there is a corresponding point $p(x)$ in $\mathpzc{P}_n$. The condition $w(p(x)) = w'(x)$ is open on $\mathrm{Hilb}^{[n]}(\mathcal{Z}_n^o/\mathpzc{P}_n)$.
	\end{lemma}
	
	\begin{proof}
		This proof is a minor modification of \cite[Lemma 4.2.2]{MR20}.
		We check our condition is preserved under generisation. Let $S$ be the spectrum of a discrete valuation ring with central fibre $0$ and generic fibre $\eta$. Suppose we have a morphism $f:S \rightarrow \mathrm{Hilb}^{[n]}(\mathcal{Z}_n^o/\mathpzc{P}_n)$ such that the image of $0$ satisfies our condition. We must show the condition holds for the image of $\eta$ as well.
		
		First consider the strata in which the image of $0$ and $\eta$ lie. Consider the  cones in the tropicalisation of $\mathcal{Z}_n^o$ corresponding to these strata denoted $F_\eta$ and $F_0$. Observe $F_\eta$ is a face of $F_0$. There are associated weighted tropical curves $(\Gamma_0,p_1,...,p_{k_0},w_0)$ and $(\Gamma_\eta,{p_1,...,p_{k_\eta}},w_\eta)$ and there is a map between the associated tropical curves $\Gamma_0' \rightarrow \Gamma_\eta'$ (constructed explicitly in \cite[Lemma 4.2.2]{MR20}). Observe for $v$ a vertex of $\Gamma_\eta$ and suppose $u_1,...,u_k$ are those vertices of $\Gamma_0$ which map to $v$ under this specialisation map. Notice: $$w(v) = \sum_{i=1}^kw(u_i).$$
		
		Now consider the dimension zero subschemes $Y_0$ and $Y_\eta$ corresponding to the image of $0$ and $\eta$ respectively. By assumption in component $u_i$ length $w(u_i)$ of $Y_0$ is supported. We fix a component of $X_{\Gamma_\eta}$ labelled by vertex $v$. Observe $Y_0$ is the closure of $Y_\eta$ in the family $f^\star(\mathcal{Z}_n^o)$. Since the closure of the component corresponding to vertex $v$ is contained in the union of the components corresponding to $u_i$ to which $v$ specialises, and this holds for all vertices of $\Gamma_\eta$ it follows that length $w(v)$ is supported on $v$.
	\end{proof}
	\begin{construction}\label{cons:PTNandunivfam}
		(The logarithmic stable pairs space $\PTnXbeta$.) The logarithmic stable pairs space $\PTnXbeta$ is defined by equipping a Deligne--Mumford stack with a logarithmic structure. The Deligne--Mumford stack is the open subset of $\mathrm{Hilb}^{[n]}(\mathcal{Z}_n^o/\mathpzc{P}_n)$ whose closed points satisfy the condition in Lemma \ref{lem:stabconditionopen}. The logarithmic structure is pulled back along the map $$\PTnXbeta \rightarrow \mathpzc{P}_n.$$
		
		We construct the universal expansion and universal stable pair. Consider the universal family for the relative Hilbert scheme of points:
		
		$$
		\begin{tikzcd}
		Y_{n,abs} \arrow[r, hook] \arrow[rd] & {\mathpzc{X}_{n,abs}} \arrow[d,"\varpi_{n,abs}"]                             \\
		& {\mathrm{Hilb}^{[n]}(\mathpzc{X}_n/\mathpzc{P}_n)}
		\end{tikzcd}.$$
		Consider the composition of inclusions $$\iota:\PTnXbeta \rightarrow  \mathrm{Hilb}^{[n]}(\mathcal{Z}_n^o/\mathpzc{P}_n)\rightarrow {\mathrm{Hilb}^{[n]}(\mathpzc{X}_n/\mathpzc{P}_n)}.$$ There is a corresponding inclusion of universal families. We define $$\mathcal{X}_{n} = \varpi_{n,abs}^{-1}\left(\iota(\PTnXbeta)\right).$$ Taking the preimage of $Y_{n,abs}$ under the corresponding morphism of universal families we obtain the following diagram: $$
		\begin{tikzcd}
		Y_n \arrow[r] & {\mathcal{X}_{n}} \arrow[d] \\
		& {\PTnXbeta}     
		\end{tikzcd}.$$ The ideal sheaf of $Y_n$ is denoted $\mathcal{J}_{Y_n}$. The universal expansion carries a universal curve $\mathcal{Z}_n$ arising from $\mathcal{Z}_n'\hookrightarrow \mathcal{X}_{n,\mathrm{PT}}$ built in Construction \ref{cons:Znprime}. The universal stable pair is
		$$[s:\mathcal{O}_{\mathcal{Z}_n}\rightarrow \mathcal{J}_{Y_n}^\star].$$
	\end{construction}

	\begin{lemma} \label{lem:PTinverse}
		Lemma \ref{lem:stablepairhilb} induces a bijection between points of $\PTnXbeta$ and stable pairs which are strongly transverse and stable.
	\end{lemma}
	
	\begin{proof}
		We write down the inverse map. To a stable pair $[s:\mathcal{O}_X\rightarrow F]$ we assign the point of $\mathrm{Hilb}^{[n]}(\mathcal{Z}_n^o/\mathpzc{P}_n)$ specified by the subscheme defined by the dual map to 
		$$q:\mathcal{O}_C \cong \mathcal{E}xt^0(\mathcal{O}_C,\mathcal{O}_C) \rightarrow \mathcal{E}xt^1(Q,\mathcal{O}_C)$$ where we define $Q = \mathrm{coker}(s)$. The dual map defines a subscheme because $q$ is surjective. Indeed we have an exact sequence $$\mathcal{O}_C \cong \mathcal{E}xt^0(\mathcal{O}_C,\mathcal{O}_C) \xrightarrow{q} \mathcal{E}xt^1(Q,\mathcal{O}_C)\rightarrow \mathcal{E}xt^1(F,\mathcal{O}_C)$$ and the third term is zero \cite[Lemma B.2]{PT2}. 
		This is clearly inverse to the map in Lemma \ref{lem:stablepairhilb}.
	\end{proof}
	
	\begin{remark}
		There is a forgetful morphism $\mathcal{PT}_{n}(X,\beta) \rightarrow \mathcal{PT}_{0}(X,\beta) $. Furthermore $$\mathcal{PT}_{0}(X,\beta) = \mathpzc{P}_{0}.$$
	\end{remark}	
	\subsection{Kool--Thomas style construction}
	The following construction identifies ${\mathcal{PT}}_n$ as the zero set of a specific section of a vector bundle inside an open subset of $\mathrm{Hilb}^{[n]}(\mathpzc{X}_n^o /\mathpzc{P}_n)$. 
	
	\begin{construction}\label{cons:KT}
		Consider the relative Hilbert scheme of points equipped with its universal family and universal subscheme
		$$ 
		\begin{tikzcd}
		{Y_{n,abs}^o} \arrow[r, hook] & {{\mathcal{X}_{n,abs}^o}} \arrow[d, "q"] \arrow[r, hook]                     & {{\mathcal{X}_{n,abs}}} \arrow[d, "q'"]                    \\
		& {\mathrm{Hilb}^{[n]}(\mathcal{X}_n^o /\mathpzc{P}_n)} \arrow[r, hook] & {\mathrm{Hilb}^{[n]}(\mathcal{X}_n /\mathpzc{P}_n)}
		\end{tikzcd}.$$
		Observe $\mathcal{X}_n^o$ comes equipped with universal curve $\mathcal{C}$ which is a divisor and pulls back to a divisor on $\mathcal{X}_{n,abs}^o$. Such a divisor induces a line bundle and canonical section $(\mathcal{O}(\mathcal{C}), s_{\mathcal{O}(\mathcal{C})})$. Restrict this vector bundle to $Y^o_{n,abs}$. Push these data down $q$ giving a sheaf $\mathcal{O}(\mathcal{C})^{[n]}$ equipped with section $s$. 
		
		The sheaf $\mathcal{O}(\mathcal{C})^{[n]}$ is locally free so defines a vector bundle. To see this first observe the morphism $q$ restricted to $Z$ is projective. It suffices to check $q'$ restricted to the universal curve in $\mathcal{X}_{n,abs}$ is projective. Restrict $q'$ morphism to the universal curve $\mathcal{C'} \hookrightarrow \mathcal{X}_{n,abs}$. This morphism is finite and the domain is proper. 
		
		We deduce $R^{\geq 1}\pi_\star(\mathcal{O}(C))=0$ by \cite[Theorem 12.11]{Hartshorne} recalling the fibre $X_y$ over every point $y$ is zero dimensional so $H^i(X_y,\mathcal{O}(C)_y)=0$. It follows that $\mathcal{O}(\mathcal{C})^{[n]}$ is locally free by \cite[Corollary 12.9 (Grauert)]{Hartshorne} The zeroes of $s$ are points of $\mathrm{Hilb}^{[n]}(\mathcal{X}_n^o /\mathpzc{P}_n)$ parametrising triples $$Y \subset C \subset X.$$ Here $Y$ is a zero dimensional subscheme of length $n$.
	\end{construction}
	
	\section{Relation to Maulik and Ranganathan's logarithmic Donaldson--Thomas spaces} The spaces $\PTnXbeta$ represent moduli functors. The moduli problems are closely related to the logarithmic Donaldson--Thomas moduli functors defined in \cite{MR20}. In this section we describe these moduli problems and prove Theorem \ref{thm:main}.
	
	We begin by recalling how to handle families of expansions. We first define an object closely related to the \textit{moduli spaces of expansions} $\mathrm{Exp}$ studied in \cite[Section~3.4]{MR20}. This moduli space of expansions carries a universal family $\pi_{\mathrm{Exp}}:\mathcal{Y}\rightarrow \mathrm{Exp}$ defined in \cite[Section~3.7]{MR20}. Fix an integer $n$ and consider the morphism from Construction \ref{cons:N} 	$$\tilde{\pi}_n:\tilde{\mathpzc{X}}_n \rightarrow \tilde{\mathpzc{P}}_{n}.$$
	Choose a splitting for the dense torus in $\tilde{\mathpzc{X}}_n$ written $\tilde{\mathpzc{X}}_{n}^o = T \oplus \mathrm{ker}(\tilde{\pi}_n)$ and denote the dense torus of $\tilde{\mathpzc{P}}_{n}$ by $(\tilde{\mathpzc{P}}_{n})^o$. Define quotient stacks $$\tilde{\mathrm{Exp}}_n = \tilde{\mathpzc{P}}_{n}/(\tilde{\mathpzc{P}}_{n})^o\quad \quad \tilde{\mathpzc{X}}_{\mathrm{Exp}_n}=\tilde{\mathpzc{X}}_n/T.$$ The morphism $\tilde{\pi}_n$ passes to the quotient. Further the symmetric group on $n$ letters acts equivalently: we further quotient by this group action to obtain $\pi_{\mathrm{Exp}_n}:\mathcal{X}_{\mathrm{Exp}_n}\rightarrow \mathrm{Exp}_n$. The diagram of quotient stacks 
	$$
	\begin{tikzcd}
	\mathcal{X}_{\mathrm{Exp}_n} \arrow[d, "\pi_{\mathrm{Exp}_n}"'] \arrow[r, "\pi_{X}"] & X \\
	\mathrm{Exp}_n                                                               &  
	\end{tikzcd}$$
	defines a map from $\mathrm{Exp}_n$ to the stack $\mathrm{Exp}$ defined by Maulik and Ranganathan along which the universal data $\mathcal{Y} \rightarrow \mathrm{Exp}$ pulls back to $\pi_{\mathrm{Exp}_n}$. Let $\underline{\mathrm{Exp}}$ be the image of this morphism and $\mathcal{X}_{\underline{\mathrm{Exp}}}$ its preimage in the universal family.

	\begin{definition} The \textit{moduli space of logarithmic stable pairs} on $X$ denoted $\mathcal{PT}_n^{\mathrm{MR}}$ is the category fibred in groupoids whose objects over $B$ are the following data. 
		\begin{enumerate}[(1)]
			\item A morphism $B\rightarrow \underline{\mathrm{Exp}}$. Pulling back $\mathcal{X}_{\underline{\mathrm{Exp}}}$ furnishes a family of expansions
			$$	
			\begin{tikzcd}
			\mathcal{X}_B \arrow[d, "\pi_B"] \arrow[r, "{\pi_{X}}"] & X \\
			B                                                         &  
			\end{tikzcd}.$$
			
			\item A stable pair $s$ in $D^b(\mathcal{X}_B)$ such that the restriction of $s$ to $\pi_n^{-1}(x)$ for each closed point $x$ in $B$ is strongly transverse and stable.
		\end{enumerate}
	\end{definition}
	\begin{theorem}\label{thm:modularity}
		There is an isomorphism of Deligne--Mumford stacks $$\PTnXbeta\tilde{\rightarrow}\mathcal{PT}_n^{\mathrm{MR}}.$$ This isomorphism identifies the universal expansion and stable pair of the domain with the expansion $\UM_{n}$ and universal stable pair from Construction \ref{cons:PTNandunivfam}.
	\end{theorem}
	
	\begin{proposition}
		Theorem \ref{thm:modularity} holds when $n=0$.
	\end{proposition}
	
	\begin{proof}
		First observe $\mathcal{PT}_0^{\mathrm{MR}}$ is logarithmically smooth and thus normal. Indeed we check the natural map $$\mathcal{PT}_0^{\mathrm{MR}}\rightarrow \mathrm{Exp}_n$$ is smooth, this implies the domain is logarithmically smooth. This map admits a perfect two term deformation obstruction theory. Consider point $y$ corresponding to $Z \rightarrow X_\Gamma \rightarrow X$ where $Z$ has ideal sheaf $J_Z$. The obstruction space associated with this point is $$\mathrm{Ext}^2(\mathcal{J}_Z,\mathcal{J}_Z)$$ which vanishes since $\mathcal{J}_Z$ is locally free. Indeed $J_Z$ is globally generated by a section $f$ - the pullback of $\mathcal{O}(1)$. The map is thus unobstructed and since $\mathrm{Exp}_n$ is logarithmically smooth, so is $\mathcal{PT}_0^{\mathrm{MR}}$.
		
		Since $\mathcal{PT}_0$ carries a universal family, the universal property of $\mathcal{PT}_0^{\mathrm{MR}}$ defines a fibre square $$
		\begin{tikzcd}
		\mathcal{X}_0 \arrow[r] \arrow[d]          & \mathcal{X}_0' \arrow[d]   \\
		{\mathcal{PT}_0} \arrow[r, "u_0"] & {\mathcal{PT}_0^{\mathrm{MR}}}
		\end{tikzcd}.$$
		
		The morphism $u_0$ descends to a map on coarse moduli spaces $$\tilde{u}_0: {PT}_0(X,\beta) \rightarrow {PT}_0(X,\beta)^\mathrm{MR}.$$ First check $\tilde{u}_0$ is an isomorphism. Indeed Theorem \ref{thm:univcurve} tells us $\tilde{u}_0$ is a bijection on points. Furthermore $\tilde{u}_0$ is birational, being an isomorphism on the dense open subscheme parameterising strongly transverse and stable curves on unexpanded $X$. Since the target is normal, by Zariski's main theorem $\tilde{u}_0$ is an isomorphism.
		
		We check $\tilde{u}_0$ induces an identification of coarse moduli  $\mathcal{X}_0'\times_{\mathcal{PT}_0^{\mathrm{MR}}} {\mathcal{PT}_0(X,\beta)} =\mathcal{X}_0 $. Combined with the map $\mathcal{X}_0' \rightarrow {\mathcal{PT}_0(X,\beta)}$ we obtain an (obviously birational) map $\mathcal{X}_0'\rightarrow \mathcal{X}_0$ which is a bijection on points. The target is normal because it was constructed as a normal toric variety so by Zariski's main theorem this map is an isomorphism on coarse moduli spaces.
		
		It remains to upgrade the isomorphism of coarse moduli spaces and universal families to an isomorphism of Deligne--Mumford stacks. This is an immediate consequence of the universal property of universal weak semistable reduction \cite{molcho2019universal}. This universal property guarentees a terminal way to upgrade $\tilde{u}_0^{-1}$ to a morphism of Deligne--Mumford stacks $u_0^{-1}$ compatible with universal families. Universality ensures $u_0^{-1}$ is the inverse of $u_0$.
	\end{proof}

	\begin{proof}[Proof of Theorem \ref{thm:modularity} for $n>0$]
		Note there are morphisms:	
		$$\mathcal{PT}_n^{\mathrm{MR}}\rightarrow \mathcal{PT}_0^{\mathrm{MR}}\xrightarrow{u_0^{-1}} \mathcal{PT}_0(X,\beta).$$
		The cokernel of the universal stable pair specifies a subscheme $Y_n'$ inside the universal curve $\mathcal{Z}_0' \rightarrow \mathcal{PT}_0(X,\beta)$. On each fibre $Y_n'$ is zero dimensional and length $n$. We obtain a map to $\PTnXbeta$ an open subset of $\mathrm{Hilb}^{[n]}(\mathcal{Z}_n^o/\mathpzc{P}_n)$ by the universal property of this relative Hilbert scheme $$u_n^{-1}:\mathcal{PT}_n^{\mathrm{MR}}\rightarrow \PTnXbeta.$$ This morphism arose from a universal property and thus is compatible with universal families. Consequently it is an inverse to $u_n$.
	\end{proof}
	
	\subsection{Proof of Theorem \ref{thm:main}}
	\begin{proof}[Proof of Theorem \ref{thm:main}]
		We start with the case $n=0$. Maulik and Ranganathan's inverse system arose from the possible polyhedral structures on our moduli space of tropical curves. Following \cite{MR20} to show $\mathcal{PT}_0(X,\beta)$ is a terminal object we must show as a cone complex every moduli space of tropical curves is a subdivision of $\mathrm{Exp}$. But combinatorial type of tropical curve must be constant on cones of any moduli space of tropical curves. No two maximal cones of $\mathrm{Exp}$ correspond to the same combinatorial type of tropical curve. 
		
		On any moduli space of weighted tropical curves the combinatorial type is constant on cones. Since there is at most one maximal cone of $\PTnXbeta$ for each combinatorial type of weighted tropical curve our theorem is proved.
	\end{proof}
	
	\begin{remark}
		Maulik and Ranganathan consider moduli spaces of tropical curves without weights. This nuance does not effect the isomorphism class of the resulting moduli space. 
	\end{remark}
	\begin{remark}\label{rem:toricStructures}
		The logarithmic structure on Maulik and Ranganathan's moduli space (which is pulled back from the logarithmic structure on their version of $\mathrm{Exp}$) differs from the toric logarithmic structure for our moduli space. In this remark we give two perspectives on the difference.
		
		Tropically the difference in logarithmic structure can be understood as follows. Points in the tropical object used to construct the \cite{MR20} version of $\mathrm{Exp}$ biject with tropical curves inside the fan of $X$. By contrast, points in the fan $|\Delta|$ do not biject with tropical curves, see Example \ref{ex:TropNotBiject}. Thus the logarithmic stratification in \cite{MR20} is more coarse than the logarithmic stratification in our construction.
		
		Another way of understanding the difference in logarithmic structure is by studying the locus where the ghost sheaf is trivial. In logarithmic linear system this locus is the dense torus and corresponds to curves in $X$ with defining equation $$\sum_{{(i,j)}\in \Delta}\lambda_{i,j} X^iY^j=0$$ where $\lambda_{i,j}$ is non-zero for all $i$ and $j$. In the \cite{MR20} space the locus with trivial ghost sheaf includes points where $\lambda_{i,j}$ is non-zero for only $(i,j)$ a corner of $\Delta$ and is thus larger.
	\end{remark}
	
	\section{Computing Euler--Satake characteristics}\label{sec:algorithm}
	In this section we compute the Euler--Satake characteristic of $\PTnXbeta$. Throughout we use $\chi$ to denote Euler--Satake characteristic with compact support. The Euler--Satake characteristic is an invariant of an orbifold closely related to the Euler characteristic of the coarse moduli space. Indeed in the case $\PTnXbeta$ is a scheme, the Euler--Satake characteristic coincides with the Euler characteristic.
	
	\subsection{The Euler--Satake characteristic for absolute Hilbert schemes of points}
	\subsubsection{Standard results from Algebraic Topology} Let $X$ be a Deligne--Mumford stack and $G$ a finite group. The Euler--Satake characteristic of a Deligne--Mumford stack which is also a scheme coincides with the usual topological Euler characteristic. The following property then characterises the Euler--Satake characteristic $$\chi([X/G])=\frac{\chi(X)}{|G|}.$$ 
	We begin with two results from algebraic topology. 
	
	\begin{lemma}\label{lem:algtop1}
		Let $X$ be a topological space containing an open set $U$ with complement $Z$. Then the following equality holds:
		\begin{equation}\label{eulerchar}
		\chi(X) = \chi(U)+\chi(Z)
		\end{equation}
	\end{lemma}
	\begin{proof}
		The orbifold $X$ admits a locally closed stratification into a finite number of sets $$X = \bigcup_{H \textrm{ a group}}Y_H.$$ Here $Y_H$ is the locally closed stratum of $X$ whose stabiliser is the group $H$. The Euler--Satake characteristic is the weighted sum $$\sum_H \frac{1}{|H|}\chi_c\left(Y_H\right).$$ Here $\chi_c$ denotes topological Euler characteristic with compact support. The Lemma follows from the corresponding statement for compactly supported Euler characteristic.
	\end{proof}
	Similarly one can check $\chi(X\times Y) = \chi(X) \times \chi(Y)$.
	\begin{lemma} \label{lem:algtop2}
		The Euler--Satake characteristic of an orbifold equipped with a free $\mathbb{C}^\star$ action is the Euler--Satake characteristic of the fixed point locus.
	\end{lemma}
	\begin{proof}
		Since group actions preserve automorphism groups our action restricts to an action of $(\mathbb{C}^\star)^n$ on $Y_H$ for each $H$. Denote the fixed locus by $Y_H^{\mathrm{fix}}$. The result follows noting $$\chi_c(Y_H) = \chi_c(Y_H^\mathrm{fix}).$$See \cite{BIALYNICKIBIRULA197399} for a proof.
	\end{proof}
	
	\subsubsection{Relative Hilbert schemes of points on a curve.} Let $C_n$ be the moduli space of $n$ distinct labelled points on a smooth curve $C$. In this situation $C$ is a scheme, not a more general orbifold, and the Euler--Satake characteristic is the usual Euler characteristic with compact support. We set $C^n$ the product of $n$ copies of $C$. Define $$D = \{(p_1,...,p_n)| p_i = p_j \textrm{ some } i \ne j\} \subset C^n.$$ Certainly we have $C^n = C_n \cup D$. We use throughout the shorthand $$[n] = \{1,...,n\}.$$ We use the notation $\stirling{n}{k}$ for Stirling numbers of the second kind. The following lemma is used repeatedly in subsequent analysis. 
	
	\begin{lemma}\label{lem:product}
		The sequence $\chi(C_n)$ satisfies the following recursion $$\chi(C_n) = \chi(C)^{n} - \sum_{k=1}^{n-1}\stirling{n}{k} \chi(C_{k}).$$
	\end{lemma}
	\begin{proof}
		The case $n=1$ is a tautology. We define a locally closed stratification $$\bigcup_{k=1}^{n} C_{n,k} = C^n$$ where $C_{n,k} = \{(p_1,...,p_n)|\mathrm{Card}(\{p_i\})=k\}.$
		Clearly $C_n=C_{n,n}$. Observe a point of $C_{n,k}$ is specified by a pair $(p,\varphi)$ where $p$ is a point of $C_k$ and $\varphi$ a surjective function $\varphi:[n]\rightarrow [k]$. Two such pairs specify the same point if and only if they differ by post--composing with an element of $S_{k}$. It follows that $$\chi(C_{n,k}) = \stirling{n}{k} \chi(C_k).$$ Sum over $k$, notice $\chi(\prod_{i=1}^nC)=\chi(C)^n$ and appeal to Lemma \ref{lem:algtop1} to complete the proof.
	\end{proof}
	
	\noindent A special case of this lemma is of particular use: let $\mathbb{C}^\star$ act on $(\mathbb{C}^\star)^n\backslash D$ diagonally and denote the quotient $C_n^\dagger$. Notice $C_n^\dagger$ is the moduli space of $n$ distinct labelled points on $(\mathbb{C}^\star)^n$ quotiented by the action of $\mathbb{C}^\star$ induced by the diagonal action.
	
	\begin{lemma}
		We have the following recursive formula for Euler characteristics $$c_n=\chi(C_n^\dagger) = (-1)^{n-1} - \sum_{k=2}^{n-1}\stirling{n-1}{k-1} \chi(C_{k}^\dagger).$$
	\end{lemma}
	
	\begin{proof}
		Observe $C_n^\dagger$ is isomorphic to  $$(\mathbb{C}^\star\backslash \{1\})^{n-1}\backslash \{(\lambda_2,...,\lambda_n) |\lambda_i = \lambda_j \textrm{ for some }i \ne j\}.$$
		and we apply Lemma \ref{lem:product} with $\chi(C)=-1$ to deduce the statement.
	\end{proof}

	\subsubsection{Rubber Hilbert schemes on thickenings of $\mathbb{C}^\star$}\label{sec:thickenings} Fix an integer $n$.
	The torus $\mathbb{C}^\star$ acts on $$\mathbb{C}^\star\times \mathrm{Spec}(\mathbb{C}[\varepsilon]/(\varepsilon^n)) = \mathrm{Spec}(\mathbb{C}[X,X^{-1},\varepsilon]/(\varepsilon^n))$$ through its action on the first component. This induces an action of the same group on the Hilbert scheme of k points $$\mathrm{Hilb}^{k}(\mathbb{C}^\star \times \mathrm{Spec}(\mathbb{C}[\varepsilon]/(\varepsilon^n))).$$ 
	We denote the quotient $$H_{k,n}=\mathrm{Hilb}^{k}(\mathbb{C}^\star \times \mathrm{Spec}(\mathbb{C}[\varepsilon]/(\varepsilon^n)))/\mathbb{C}^\star.$$ Note $H_{k,n}$ is in general a Deligne--Mumford stack not a scheme. The Euler--Satake characteristic of $H_{k,n}$ will be expressed in terms of $$P_q(\ell) = \left|\left\{\mathrm{Partitions } \sum_j \ell_j = \ell \Bigg| \ell_j<q \textrm{ for all }j\right\}\right|.$$
	For a tuple $\underline{\ell} = (\ell_1,...,\ell_{|\ell|})$ of natural numbers we define $N(\underline{\ell})$ to be the product over $k$ in $\mathbb{N}$ of $n(\underline{\ell},k)$ where $n(\underline{\ell},k) = d!$ where $d$ is the number of times $k$ appears in the tuple $\underline{\ell}$. Let $\mathrm{Part}(k)$ be the set of unordered partitions. 
	\begin{lemma} We have the following equality,
		$$h_{k,n}=\chi(H_{k,n}) = \sum_{\underline{\ell}\in \mathrm{Part}(k)}\frac{\chi(C_{|\underline{\ell}|}^\dagger)}{N(\underline{\ell})} \prod_{\ell_j \in \underline{\ell}} P_n(\ell_j).$$
		The product is over $\ell_j$ in the partition $\underline{\ell} = (\ell_1,...,\ell_{|\ell|})$ where $\sum_j \ell_j = k$.
	\end{lemma} 
	
	\begin{proof}
		A point of $H_{k,n}$ specifies a partition of $k$. Indeed such a point specifies a closed subscheme $Z$ of $\mathbb{C}^\star\times \mathrm{Spec}(\mathbb{C}[\varepsilon]/(\varepsilon^n))$ of dimension zero. This is the data of points $\{1,...,|\ell|\}$ and at each point the stalk of the structure sheaf $\mathcal{O}_Z$ has length $\ell_i$ where $\sum_i \ell_i = n$. The quotient $H_{k,n}$ admits a locally closed stratification $$H_{k,n} = \bigcup_{\underline{\ell}}H_{k,n}^{\underline{\ell}}$$ where $H_{k,n}^{\underline{\ell}}\subset H_{k,n}$ are moduli spaces of subschemes which specify partition $\underline{\ell} = (\ell_1,...,\ell_{|\ell|})$. Notice $H_{k,n}^{\underline{\ell}}$ is isomorphic as a stack to the quotient of the scheme $$C_{|\ell|}^\dagger\times \prod_{\ell_i\in \underline{\ell}} \mathrm{Hilb}^{\ell_i}(\mathrm{Spec}(\mathbb{C}[X]/(X^{\ell_i})) \times \mathrm{Spec}(\mathbb{C}[\varepsilon]/(\varepsilon^n)))$$ by the free action of a product of symmetric groups. It thus suffices to compute the Euler characteristic of this space. This product of symmetric groups has order $N(\ell)$.
		There is a $(\mathbb{C}^\star)^2$ action on $$\mathrm{Spec}(\mathbb{C}[X,\varepsilon]/(X^{\ell_i},\varepsilon^n))$$ specified on the level of coordinate rings by $$(t_1,t_2)(f(X,\epsilon)) = f(t_1^{-1}X,t_2^{-1}\epsilon).$$ The Hilbert scheme of points $$\mathrm{Hilb}^{\ell_i}(\mathrm{Spec}(\mathbb{C}[X]/(X^{\ell_i})) \times \mathrm{Spec}(\mathbb{C}[\varepsilon]/(\varepsilon^n)))$$ inherits an action whose fixed points are monomial ideals. Counting monomial ideas we deduce that $$\chi\left(\mathrm{Hilb}^{\ell_i}(\mathrm{Spec}(\mathbb{C}[X]/(X^{\ell_i})) \times \mathrm{Spec}(\mathbb{C}[\varepsilon]/(\varepsilon^n)))\right)=P_n(\ell_i)$$ and the result follows.
	\end{proof}
	
	\subsubsection{Curve moduli}\label{sec:curvemoduli} Let $X_\sigma$ be the toric variety whose fan is the star fan of $v_0$ depicted in Figure~\ref{fig:buildingblock} and let $\beta$ be the curve class in $A_\star(X_\sigma)$ whose intersection number with horizontal boundary divisors is as indicated by $i$ and $j$ and which intersects the non--horizontal boundary divisors in a single reduced point. 
	
	Let $T'^{i,j}$ be the moduli space of strongly transverse curves in $X_\sigma$ with Chow class $\beta$. A strongly transverse curve is specified by its restriction to the dense torus. This restriction is cut out by a single polynomial in $\mathbb{C}[X^{\pm 1},Y^{\pm 1}]$. Many polynomials give rise to the same subscheme, but there is a unique such polynomial $f'(X,Y)$ in $\mathbb{C}[X,Y]$ such that neither $X$ nor $Y$ is a factor. We study polynomials of this form. Knowing how the curve intersects the toric boundary means we can uniquely write $$f'(X,Y)=f(Y)(g(Y)+ Xh(Y)) $$ where $g$ and $h$ are coprime elements of $\mathbb{C}[X,Y]$, non-zero and have no monomial factors. Certainly we have the equalities $$\mathrm{deg}(f)+\mathrm{deg}(g) = i\textrm{ and } \mathrm{deg}(f)+\mathrm{deg}(h) = j.$$
	
	The quotient of $T'^{i,j}$ by the action of the dense torus of $X_\sigma$ is denoted $T^{i,j}$. Stratify $$T^{i,j} = \bigcup_{\alpha\in A(i,j)} T^{i,j}_\alpha.$$ The discrete data $\alpha$ comes in two parts. First a set $\alpha^+$, which is the set of roots of $fg$, equipped two integer valued functions $$m_g,m_f^+: \alpha^+ \rightarrow \mathbb{R}$$ which records the multiplicity of each element of $\alpha^+$ as a root of $f$ and $g$. Second a set $\alpha^-$ which we think of as the set of roots of $fh$, equipped with multiplicity functions $m_h,m_f^-$. 
	
	Denote the number $v$ of roots of $g$ which are not roots of $f$, the number $w$ of roots of $h$ which are not roots of $f$ and the number $s$ of distinct roots of $f$. These are functions of $\alpha$. Write $t_\alpha =s+v+w$. Note $t_\alpha$ is the number of $Y$ coordinates at which the zero set of $f'$ hits two toric boundaries. The multiplicity functions specify partitions of the set of roots. This is a partition of $t_\alpha = \sum_i m_i$. The partition depends only on $\alpha$ and we write $R(\alpha) = \prod_i (m_i!)$. 
	
	Since a polynomial is characterised by its roots we identify $T^{i,j}_{\alpha} ={C}_{t_\alpha}^\dagger/G$ where $G$ permutes roots with the same multiplicity and thus has order $R(\alpha)$. We deduce $$\chi(T^{i,j}_\alpha) = \chi({C}_{t_\alpha})/R(\alpha).$$ We have already computed the numerator.
	
	\subsubsection{Moduli of points on a curve.} We denote the moduli space of $k$ points on a fixed strongly transverse and stable curve in $X_\sigma$, say $$C=V(f(Y)(g(Y) + X h(Y)))\subset (\mathbb{C}^\star)^2,$$ by $M_{\alpha,k}$. Here $\alpha$ records the same discrete data about $C$ as in the last section.
	The curve $C$ has an irreducible component $C_\alpha = V(g(Y) + Xh(Y))$ isomorphic to $\mathbb{P}^1$ less $b$ points where $$b =2+ \textrm{number of roots of g} + \textrm{number of roots of h}.$$ In this equation roots are counted without multiplicity. The other irreducible components are all isomorphic to $\mathbb{C}^\star$. The Hilbert scheme of length $k$ subschemes of $C_\alpha$ is denoted $C_\alpha^{[k]}$. The space of $|\ell|$ distinct points on $C_\alpha$ is denoted $(C_\alpha)_{|\ell|}$.
	
	Performing the analysis described in Section \ref{sec:thickenings} we find 
	$$h_{k,l}=\chi(C_\alpha^{[k]})= \sum_{\underline{\ell} \in \mathrm{Part}(k)} \frac{\chi\left((C_\alpha)_{|\underline{\ell}|}\right)}{N(\underline{\ell})}.$$ Since $\mathbb{C}^\star$ acts freely on every irreducible component of $C$ other than $C_\alpha$, it follows from \cite[Equation (3.3)]{KoolShendeThomas} that $$\chi(M_{\alpha,k}) = \chi(C_\alpha^{[k]}).$$ 
	
	\subsection{Euler--Satake characteristic of multiplicands} Our approach is to compute the Euler--Satake characteristic of a logarithmic stratification of $\PTnXbeta$. To compute this we express each logarithmic stratum as a product of moduli spaces. In this subsection we compute the Euler--Satake characteristic of each multiplicand.
	\subsubsection{Defining a multiplicand}\label{sec:multiplicand}
	Figure \ref{fig:buildingblock} depicts part of a tropical curve $\Gamma_\varphi$ (red and black) superimposed on the fan of $\mathbb{P}^1 \times \mathbb{P}^1$ to form the pre--expansion tropical curve $\Gamma$. We denote the part of $\Gamma$ appearing in Figure \ref{fig:buildingblock} by $\Gamma^\star$. Consider decorating this diagram with $k$ marked points $p_1,...,p_k$ all of which lie on the red line. Assume these marked points are equipped with a weight $w(p_i)$ such that $\sum_i w(p_i)=n$. Add a vertex to $\Gamma^\star$ at each point $p_s$ to form a tropical curve $\Gamma'$.
	
	Observe $\Gamma'$ defines an expansion $X_{\Gamma'}$ of $X$ via Construction \ref{cons:expansions} and there is a natural (but not surjective) morphism $X_{\Gamma'}\rightarrow X_{\sigma}$. Let $W$ be the space of strongly transverse and stable curves in $X_\sigma$ which intersect the black lines in one reduced point and the red lines in subschemes of length $i$ and $j$ respectively. Pulling curves back along the map $X_{\Gamma'}\rightarrow X_{\sigma}$, we think of $W$ as a moduli space of strongly transverse and stable curves in $X_{\Gamma'}$ (up to isomorphism). In this capacity there is a universal diagram
	$$
	\begin{tikzcd}
	{C} \arrow[r, hook]\arrow[rd] & {X}_{\Gamma'}\times W \arrow[r] \arrow[d] & X \\
	& W                               &  
	\end{tikzcd}.
	$$
	
	We define a moduli space $E_{(\Gamma^\star,\{p_i\},w)}'$ to be the locally closed subscheme of the relative Hilbert scheme of $n$ points $\mathrm{Hilb}^{[n]}(C/W)$ such that length $w(p_i)$ is supported in the interior of the logarithmic stratum of $X_{\Gamma'}$ associated to the vertex at point $p_i$. Finally set $E_{(\Gamma^*,\{p_i\},w)}$ to be the quotient of $E_{(\Gamma^*,\{p_i\},w)}'$ by the group of automorphisms of $X_{\Gamma'}$ over $X_\Gamma$. These automorphisms are exactly the horizontal one dimensional torus action on each new component.
	
	\subsubsection{The Euler--Satake Characteristic of $E_{(\Gamma^\star,\{p_i\},w)}$.} 
	Stratify $E_{(\Gamma^\star,\{p_i\},w)}$ according to the topological type of the curve $C'$. We write $$ E_{(\Gamma^\star,\{p_i\},w)} = \bigcup E_{(\Gamma^\star,\{p_i\},w)}^\alpha.$$ This coincides with the pull back of the stratification $T_{i,j}^\alpha$ from Section \ref{sec:curvemoduli}. 
	
	Such strata are locally closed so it suffices to compute the Euler--Satake characteristic of each stratum $E_{(\Gamma^\star,\{p_i\},w)}^\alpha$. If one of the $p_i$ coincides with $v_0$ then denote this marked point $p_0$. If there is no such marked point set $w(p_0)=0$. Denote the remaining marked points by $\{p_1,...,p_k\}$. 
	
	Now set $\mathrm{Lab}(\{p_i\},\alpha)$ to be the set of functions $f:\{p_i\}\rightarrow \alpha^+ \cup \alpha^-$ where $f(p_i)$ lies in $ \alpha^+$ if and only if $p_i$ appears to the right of $v_0$.
	
	\begin{lemma}\label{lem:buildingblock} With the above notation and assuming $\{p_i\}\backslash \{p_0\}\ne \emptyset$ and $i,j$ not both zero we have $$\chi(E_{(\Gamma^\star,\{p_i\},w)}^\alpha) = \chi\left(T_{i,j}^\alpha\times \left(\bigcup_{f\in \mathrm{Lab}(\{p_i\},\alpha)} \prod_{q=1}^k H_{w(p_q),m(f(p_q))} \right)\times M_{\alpha,w(p_0)}\right).$$ 
		Otherwise we have 
		$$\chi(E_{(\Gamma^\star,\{p_i\},w)}^\alpha) = \chi\left(T_{i,j}^\alpha\times M_{\alpha,w(p_0)}\right).$$
		Since the Euler--Satake characteristic is additive over a locally closed stratification and multiplicative over cartesian products the analysis of the previous sections allows us to read off these numbers.
	\end{lemma}
	\begin{proof}
		Connected components of $E_{(\Gamma^\star,\{p_i\},w)}^\alpha$ are choices of how to allocate the $w(p_i)$ points in the component corresponding to vertex $v_i$ among the roots $\alpha^+$ (if $p_i$ to the right of $v_0$) or $\alpha^-$ (if $p_i$ to the left of $v_0$). There is a free action of $\mathbb{C}^\star$ on those components where not all $w(p_i)$ points are assigned to the same root, so we discard them from our union.
	\end{proof}
	Euler--Satake characteristic is additive and multiplicative; applying our earlier analysis computes $\chi(E_{(\Gamma^\star,\{p_i\},w)}^\alpha).$
	\begin{figure}
		\includegraphics[width=100mm]{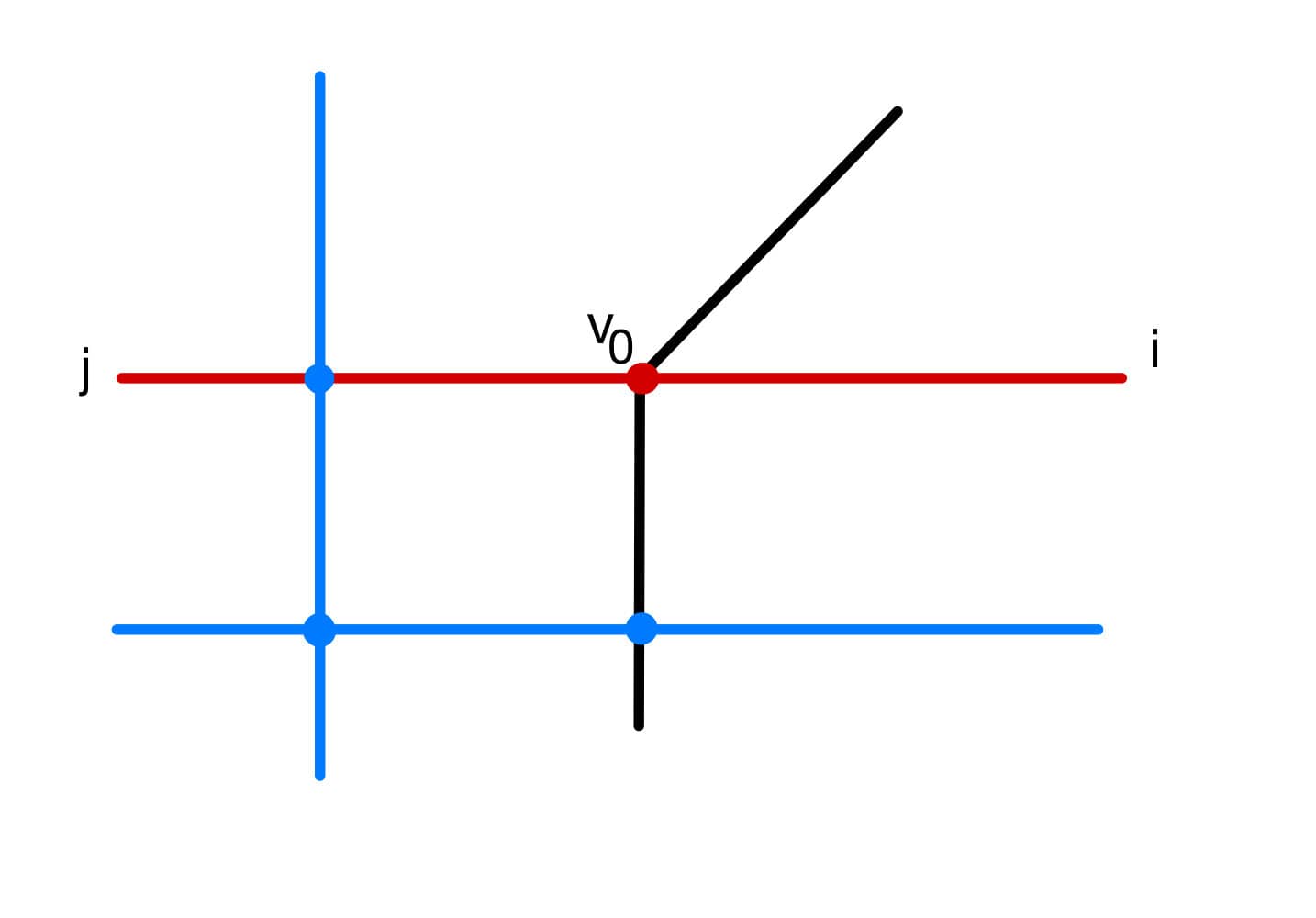}
		\caption{Tropical curve segment $\Gamma^\star$ used in our analysis. Blue is the fan of $\mathbb{P}^1 \times \mathbb{P}^1$. Black and red is the tropical curve segment superimposed on this fan.}
		\label{fig:buildingblock}
	\end{figure}
	
	\subsection{Euler--Satake Characteristic of $\mathcal{PT}_n(\mathbb{P}^1 \times \mathbb{P}^1,(1,d))$} There is a forgetful morphism (of Deligne--Mumford stacks) $$\pi_n:\mathcal{PT}_n(\mathbb{P}^1 \times \mathbb{P}^1,(1,d)) \rightarrow \mathcal{PT}_0(\mathbb{P}^1 \times \mathbb{P}^1,(1,d)).$$
	The Euler--Satake characteristic of $\mathcal{PT}_n(\mathbb{P}^1 \times \mathbb{P}^1,(1,d))$ is a weighted count of those faces in the fan $\VTM$ corresponding to strata which are not acted upon freely by any subtorus of $(\mathbb{C}^\star)^2\subset \mathbb{P}^1 \times \mathbb{P}^1$. The weight is the Euler--Satake characteristic of the fibre of $\pi_n$. We compute the weight associated to a fixed cone $\sigma$ in $\VTM$. This weight is the Euler--Satake characteristic of the locally closed subscheme $\pi_n^{-1}(O_\sigma)$ where $O_\sigma$ is the locally closed subscheme of $\mathcal{PT}_0(X,\beta)$ corresponding to $\sigma$.
	
	The cone $\sigma$ specifies the data of the combinatorial type of a pre-expansion tropical curve $\Gamma$. The logarithmic strata in the preimage in $\mathcal{PT}_n(\mathbb{P}^1 \times \mathbb{P}^1,(1,d))$ are indexed by combinatorial types of choices of $k$ (unlabelled, possibly coinciding) points on the tropical curve $\Gamma$. Continuing notation from Section~\ref{sec:HighereulerChar} we write this data as a triple $${\Gamma_k} = (\Gamma,\{p_1,...,p_k\},w).$$ Such a logarithmic stratum admits the free action of a torus and has weight zero unless: \begin{enumerate}[(1)]
		\item The points $p_i$ are disjoint from the extended main component.
		\item The tropical curve $\Gamma$ defines a stratum of $\mathcal{PT}_0(\mathbb{P}^1 \times \mathbb{P}^1,(1,d))$ on which no subtorus of $$(\mathbb{C}^\star)^2\subset \mathbb{P}^1 \times \mathbb{P}^1$$ acts freely.
	\end{enumerate} We think of $\Gamma$ as obtained by superimposing the fan of $X$ onto a tropical curve $\Gamma_\varphi$. The second hypothesis can only hold if no vertex of $\Gamma_\varphi$ lies on the extended main component. It remains to compute the Euler--Satake characteristic of such strata. 
	
	\begin{construction}
		Fix the triple $${\Gamma_k}=(\Gamma,\{p_1,...,p_k\},w)$$ and let $y$ a coordinate such that $\Gamma$ has a vertex with vertical coordinate $y$. We construct a triple $(\Gamma^\star_y,\{p_i^\star\},w_y)$ of the form studied in Section \ref{sec:multiplicand}.
		
		Observe $\Gamma$ has two vertices at height $y$. One vertex at $(0,y)$ is in the extended main component; we label the other vertex $v_0$. As part of the data of $\Gamma$, the horizontal edges adjacent to $v_0$ carry weights: say $i$ in the direction of the other vertex and $j$ away from it. The curve segment $\Gamma^\star_y$ is then defined to be the curve depicted in Figure \ref{fig:buildingblock} with this value of $i$ and $j$. The points $\{p_q^\star\}$ are points on $\Gamma^\star_y$. These points biject with points $\{p_q\}$ with $y$ coordinate $y$ and are placed in the obvious order. Finally $w_y(p_q^\star) = w(p_q)$.
	\end{construction}
	
	Let $\mathrm{Lin}(\Gamma)$ be the set of non-zero $y$ coordinates at which $\Gamma$ has more than one vertex. Let $\mathrm{Pt}(\Gamma)$ be the set of $y$ coordinates at which $\Gamma$ has one vertex.
	
	\begin{lemma} The torus fixed locus in the logarithmic stratum $O(\Gamma, \{p_i\},w)$ associated to tropical data $(\Gamma, \{p_i\},w)$ is isomorphic to a product of moduli spaces:
		$$O(\Gamma, \{p_i\},w)=\prod_{y\in \mathrm{Lin}(\Gamma)} E_{(\Gamma^\star_y,\{p_i\},w^\star_y)} \times \prod_{y\in \mathrm{Pt}(\Gamma)} H_{1,w(p_i)}$$
		In the final product $p_i$ is the unique point with vertical coordinate $y$.
	\end{lemma}
	\begin{proof}
		We say an irreducible component of $X_\Gamma$ is at height $y$ if it arises from a vertex with vertical coordinate $y$. Note $O(\Gamma, \{p_i\},w)$  is a moduli space of subschemes in $X_{\Gamma'}$. Consequently $O(\Gamma, \{p_i\},w)$ is a product over $y$ of moduli spaces of subschemes of components of $X_{\Gamma'}$ of height $y$. A priori one should take a product over a base $B$ which is the moduli space of subschemes of the intersection of components at various heights.  Since we are working up to the automorphism group of $X_\Gamma$ over $X$ we find $B =\mathrm{Spec}(\mathbb{C})$. This is the statement that given a subscheme at height $y_1$ and height $y_2$ there is a unique (up to rubber action) way of gluing these subschemes together. 
	\end{proof}
	\subsection{A formula}
	We now summarise our results from this section. 
	
	\begin{theorem}\label{thm:EulerChar}
		The Euler--Satake characteristic of $\PTnXbeta(\mathbb{P}^1\times \mathbb{P}^1,(1,d))$ is a weighted count of cones in the fan $|\Delta|$. The weight assigned to cone $\sigma$ maximal among cones with associated pre-expansion tropical curve $\Gamma$ is 		$$U_{n,d}(\Gamma)=\sum_{(\Gamma,\{p_i\},w)}\prod_{y_i\in \mathrm{Pt}(\Gamma')}h_{1,w(p_i)}\prod_{y\in \mathrm{Lin}(\Gamma')}\sum_{\alpha\in \Gamma(y)} \left( \chi(E^\alpha_{(\Gamma^\star_y,\{p_i\},w^\star_y)}) \right).$$
		The Euler Characteristic of $E^\alpha_{(\Gamma^\star_y,\{p_i\},w^\star_y)}$ is computed from Lemma \ref{lem:buildingblock}. We define terms in this equation.
		\begin{enumerate}[(1)]
			\item The first sum is over ways of decorating $\Gamma$ with up to $n$ marked points $\{p_i\}$ disjoint from the extended main component, each with positive integer weight $w(p_i)$ such that $$\sum_i w(p_i)=n.$$
			\item We define integers $c_n$ by requiring they satisfy the recursion
			$$c_n = (-1)^{n-1} - \sum_{k=2}^{n-1}\stirling{n-1}{k-1} c_k.$$
			\item The rational numbers $h_{k,n}$ may be computed through the following formula
			$$h_{k,n} = \sum_{\underline{\ell}\in \mathrm{Part}(k)}\frac{c_{|\underline{\ell}|}}{N(\underline{\ell})} \prod_{\ell_j \in \underline{\ell}} P_n(\ell_j).$$
			The notation in this formula is explained in Section \ref{sec:thickenings}.
		\end{enumerate}
		Weight zero is assigned to all other cones.
	\end{theorem}
	\bibliographystyle{alpha}
	\bibliography{Bibliography2}
	
\end{document}